\documentclass{article}[11pts]

\usepackage{amsthm,amsfonts,amsmath,amssymb}
\usepackage{pxfonts}
\usepackage{euscript}
\usepackage{bbold,bbm}
\usepackage[utf8]{inputenc}
\usepackage{hyperref}
\hypersetup{
    colorlinks,
    linkcolor={red!50!black},
    citecolor={blue!50!black},
    urlcolor={blue!80!black}
}
\usepackage{graphics}
\usepackage{epstopdf} % for pdflatex
\usepackage{xypic}
\usepackage{float}
\usepackage[all,2cell,arrow,matrix]{xy} \UseAllTwocells \SilentMatrices
\usepackage{graphicx}
\usepackage{comment}
\usepackage[most]{tcolorbox}
\usepackage{mathtools} 
\usepackage{tikz-cd}
\graphicspath{{graphpdf/}}

\newcommand\void[1]       {}

\newtheorem{thm}{Theorem}[section]
\newtheorem{lem}[thm]{Lemma}
\newtheorem{prop}[thm]{Proposition}
\newtheorem{cor}[thm]{Corollary}

\newtheorem{prop-defn}[thm]{Proposition-Definition}

\theoremstyle{definition}
\newtheorem{defn}[thm]{Definition}

\newtheorem{expl}[thm]{Example}
\newtheorem{rem}[thm]{Remark}

\theoremstyle{remark}

\numberwithin{equation}{section}
\numberwithin{thm}{section}

\newcommand{\be}{\begin{equation}}
\newcommand{\ee}{\end{equation}}
\newcommand{\bnu}{\begin{enumerate}}
\newcommand{\enu}{\end{enumerate}}

\newcommand{\Cb}{\mathbb{C}}

\newcommand\CA           {\EuScript{A}}
\newcommand\CB           {\EuScript{B}}
\newcommand\CC           {\EuScript{C}}
\newcommand\CD           {\EuScript{D}}
\newcommand\CE          {\EuScript{E}}

\newcommand\CL          {\EuScript{L}}
\newcommand\CM          {\EuScript{M}}
\newcommand\CN         {\EuScript{N}}
\newcommand\CO         {\EuScript{O}}
\newcommand\CP         {\EuScript{P}}
\newcommand\CQ         {\EuScript{Q}}

\newcommand\CU         {\EuScript{U}}
\newcommand\CV         {\EuScript{V}}
\newcommand\CW         {\EuScript{W}}
\newcommand\CX         {\EuScript{X}}
\newcommand\CY         {\EuScript{Y}}
\newcommand\CZ         {\EuScript{Z}}

\newcommand{\FZ}{\mathfrak{Z}}

 \DeclareMathOperator{\Hom}{Hom}
 
 \DeclareMathOperator{\Aut}{Aut}

 \DeclareMathOperator{\Ker}{Ker}

 \DeclareMathOperator{\Id}{id}
 
 \DeclareMathOperator{\ev}{ev}

 \DeclareMathOperator{\Pic}{Pic}
 \DeclareMathOperator{\chara}{char}

 \DeclareMathOperator{\Fun}{Fun}

 \DeclareMathOperator{\Alg}{Alg}
 \DeclareMathOperator{\CAlg}{CAlg}
 \DeclareMathOperator{\Mod}{Mod}

\newcommand{\bk}{{\mathbb{k}}}
\newcommand{\id}{{\mathrm{id}}}

\newcommand{\nn}{\nonumber \\}
\newcommand{\one}{\mathbf{1}}
\newcommand{\rev}{\mathrm{rev}}
\newcommand{\op}{\mathrm{op}}

\newcommand{\mtc}{\EuScript{MT}\mathrm{en}^{\mathrm{ind}}}
\newcommand{\btc}{\EuScript{BT}\mathrm{en}}

\newcommand{\mfus}{\EuScript{MF}\mathrm{us}^{\mathrm{ind}}}

\newcommand{\bfuscl}{\EuScript{BF}\mathrm{us}^{\mathrm{cl}}}
\newcommand{\ncbfuscl}{{^{\mathrm{nc}}}\EuScript{BF}\mathrm{us}^{\mathrm{cl}}}
\newcommand{\bdy}{\mathrm{bdy}}
\newcommand{\bulk}{\mathrm{bulk}}

\newcommand{\mref}[1]{(\ref{#1})}

\newcommand{\fCat}{\mathbf{Cat^{f}}}

\newif\ifFinal
\Finalfalse
\ifFinal
    \excludecomment{hide}
\else
    \includecomment{hide}
\fi

\begin{document}

%LK: added this for diagrams:
	\UseAllTwocells

\begin{center} \LARGE
%Categories of topological orders and boundary-bulk relation
Pointed Drinfeld center functor
%of taking centers of higher categories
%with values given by the center of higher categories
%between higher categories
%and the functoriality of center
%between higher categories
% theory Center functor in higher category theory realized as a holographic boundary to bulk map in topological order
\\
%The boundary to bulk map in topological order\\ as the center functor between higher categories
%(??? I still think the second one is more clear. One do not know what you try to say from the first version)
\end{center}

\vskip 1.5em
\begin{center}
{\large
Liang Kong$^{a}$,\,
Wei Yuan$^{b,c}$,
Hao Zheng$^{a,d}$\,
~\footnote{Emails:
{\tt  kongl@sustech.edu.cn, wyuan@math.ac.cn, hzheng@math.pku.edu.cn}}}
\\[2em]
$^a$ Shenzhen Institute for Quantum Science and Engineering,\\
and Department of Physics,\\
Southern University of Science and Technology, Shenzhen 518055, China 
\\[1em]
$^b$ Academy of Mathematics and Systems Science,\\
Chinese Academy of Sciences, Beijing 100190, China
\\[1em]
$^c$ School of Mathematical Sciences,\\
University of Chinese Academy of Sciences, Beijing 100049, China
\\[1em]
$^d$ Department of Mathematics,\\
Peking University, Beijing, 100871, China
\end{center}

\vskip 3em

\begin{abstract}
In this work, using the functoriality of Drinfeld center of fusion categories, we generalize the functoriality of full center of simple separable algebras in a fixed fusion category to all fusion categories. This generalization produces a new center functor, which involves both Drinfeld center and full center and will be called the pointed Drinfeld center functor. We prove that this pointed Drinfeld center functor is a symmetric monoidal equivalence. It turns out that this functor provides a precise and rather complete mathematical formulation of the boundary-bulk relation of 1+1D rational conformal field theories (RCFT). In this process, we solve an old problem of computing the fusion of two 0D (or 1D) wall CFT's along a non-trivial 1+1D bulk RCFT. 
\end{abstract}

\tableofcontents

\section{Introduction} \label{sec:introduction}

%Classically, it was known that the notion of center of an algebra is not functorial because an algebra homomorphism between two algebras does not induce an algebra homomorphism between their centers. It turns out that this naive observation can be misleading. 

Motivated by the theory of 1+1D rational conformal field theories (RCFT), it was shown in \cite{dkr} that the notion of center can be made functorial if we define the morphisms in the domain and codomain categories properly. More precisely, for two algebras $A, B$ and an $A$-$B$-bimodule $M$, the following assignment defines a lax functor $Z$: 
\be \label{eq:drinfeld-center-functor-0}
( A \xrightarrow{M} B )
\quad \longmapsto \quad
( Z(A) \xrightarrow{Z^{(1)}(M):=\hom_{A|B}(M,M)} Z(B) ),
\ee
where $Z(A)$ denotes the center of $A$ and $\hom_{A|B}(M,M)$ denotes the algebra of $A$-$B$-bimodule maps of $M$. Moreover, $Z$ is a functor if $A,B$ are restricted to simple separable algebras. 

This naive result becomes highly non-trivial if we consider algebras in a fusion category $\CC$. The main result of \cite{dkr} showed that the same assignment \eqref{eq:drinfeld-center-functor-0} is still a well-defined functor if $A$ and $B$ are simple separable algebras in $\CC$ and $Z(A)$ is the full center of $A$ in the Drinfeld center of $\CC$ \cite{davydov}. When $\CC$ is the modular tensor category $\Mod_V$ of modules over a rational vertex operator algebra $V$ (VOA) \cite{huang-mtc2}, this functor provides a precise mathematical description of the boundary-bulk relation of RCFT's with the fixed chiral symmetry $V$. 

\smallskip
In order to generalize this relation to different chiral symmetries, we need to prove certain functoriality of Drinfeld center, which was established by two of the authors in \cite{kz1}. More precisely, for two fusion categories $\CL, \CM$ and a semisimple $\CL$-$\CM$-bimodule $\CX$, it was proved that the following assignment
\be \label{eq:drinfeld-center-functor-1}
( \CL \xrightarrow{\CX} \CM )
\quad \longmapsto \quad
( \FZ(\CL) \xrightarrow{\FZ^{(1)}(\CX):=\Fun_{\CL|\CM}(\CX,\CX)} \FZ(\CM) ),
\ee
where $\FZ(\CC)$ denotes the Drinfeld center of $\CC$ and $\Fun_{\CL|\CM}(\CX,\CX)$ denotes the category of $\CL$-$\CM$-bimodule functors, gives a well-defined functor. In fact, this Drinfeld center functor provides a precise and complete mathematical description of the boundary-bulk relation of 2+1D (anomaly-free) topological orders with gapped boundaries. In particular, $\FZ^{(1)}(\CX)$ describes a 1+1D gapped domain wall between the two 2+1D topological orders defined by $\FZ(\CL)$ and $\FZ(\CM)$. 
%In Section \ref{sec:kz-thm}, we will briefly discuss how to enhance it to a conjectural Drinfeld center 2-functor (or even a 3-functor), the image of which provides a mathematical description of 2+1D non-chiral topological orders with gapped defects of all codimensions. 

\smallskip
To generalize the boundary-bulk relation of RCFT's with different chiral symmetries demands us to combine the full center functor with the Drinfeld center functor to a so-called {\em pointed Drinfeld center functor} as illustrated in the following assignment: 
$$  %\label{eq:ZZ-functor}
\small
\left( \xymatrix{ (\CL,L) \ar[rr]^{(\CX,\, x)} & & (\CM,M)} \right)
\longmapsto 
\left( \xymatrix{ (\FZ(\CL),Z(L)) \ar[rrr]^{(\FZ^{(1)}(\CX),\,\ Z^{(1)}(x))} & & & (\FZ(\CM),Z(M))} \right),
$$
for indecomposable multi-tensor categories $\CL,\CM$ and simple exact algebras $L\in\CL, M\in\CM$. The pair $(\FZ(\CL),Z(L))$ is indeed the categorical center of $(\CL,L)$ as shown in \cite{s2}. We prove in Section \ref{sec:Df-center-functor} (see Theorem \ref{thm:Df-center}) that it is a well-defined functor. When we restrict the domain to indecomposable multi-fusion categories and simple separable algebras, this pointed Drinfeld functor becomes a monoidal equivalence (see Theorem \ref{thm:ff}). This monoidal equivalence provides a precise mathematical description of the boundary-bulk relation of 1+1D RCFT's with different (but still rational) chiral symmetries. It also summarizes and generalizes a few earlier results in the literature. In the process of proving this functoriality, we prove a key formula (\ref{eqn:main}), which solves an old open problem of defining and computing the fusion of two 1D (or 0D) domain walls along a non-trivial 1+1D bulk RCFT. We postpone a detailed explanation of the physical significants of this work until Section\,\ref{sec:cft}.

\smallskip
In Section \ref{sec:cft}, we will also briefly discuss how to generalize this 1-functor to a 2-functor (even further to a 3-functor) as illustrated by the following assignment: 
$$  %\label{eq:ZZ-functor}
\small
\xymatrix{ (\CL,L) \ar@/_12pt/[rr]_{(\CY,\, y)} \ar@/^12pt/[rr]^{(\CX,\, x)} & \Downarrow (F,f) & (\CM,M)}
\longmapsto 
\xymatrix{ (\FZ(\CL),Z(L)) \ar@/_16pt/[rr]_{(\FZ^{(1)}(\CY),\,Z^{(1)}(y))} \ar@/^16pt/[rr]^{(\FZ^{(1)}(\CX),\,Z^{(1)}(x))} & \Downarrow (\FZ^{(2)}(F),\,Z^{(2)}(f)) & (\FZ(\CM),Z(M))}
$$
where $f:F(x)\to y$ is a morphism in $\CY$ and $(\FZ^{(2)}(F),Z^{(2)}(f))$ will be defined in Section \ref{sec:df-3equiv}. The image of this 2-functor provides a precise mathematical description of 1+1D RCFT's with topological defects of all codimensions. 

\smallskip
Mathematically, the main theme of this work lies in the intricate interrelations between algebras in different dimensions, i.e. $E_0,E_1,E_2$-algebras (see for example \cite{lurie}). We reveal this interrelation in our setting in details in Section \ref{sec:left-right-full-centers} when we introduce the notions of the left/right/full center of an algebra in a monoidal category. The most technical part of this work lies in proving the key formula (\ref{eqn:main}), which compute the relative tensor product of two $E_0$-algebras (or $E_1$-algebras) over an $E_2$-algebra. Algebras in different dimensions interacting with each other in a unified framework is one of the central themes of mathematical physics in our time \cite{lurie,af}. The main result of this work is a manifestation of this theme.

\smallskip
We briefly explain the layout of this paper. In Section \ref{sec:fusion-cat}, we review relevant results in tensor categories and set the notations along the way. In Section \ref{sec:left-right-full-centers}, we introduce the notion of a left/right/full center of an algebra in a monoidal category, and explain their relation to internal homs, and prove some of its properties. These notions play crucial roles in our construction of the pointed Drinfeld center functor. In Section \ref{sec:Drinfeld-full-center-functor}, we construct the pointed Drinfeld center functor, and prove that it is a well-defined symmetric monoidal equivalence. Section \ref{sec:formula} is devoted to prove the key formula (\ref{eqn:main}) in Theorem \ref{thm:main_formula}. The left side of (\ref{eqn:main}) defines the horizontal composition of 1-,2-morphisms in the codomain category of the pointed Drinfeld center functor; the right side is the image of the horizontal composition of 1-,2-morphisms in the domain category; the isomorphism in this formula guarantees the functoriality. In Section \ref{sec:cft}, we sketch the pointed Drinfeld center 3-functor, and provide the physical motivations and meanings of the 1-truncation and the 2-truncation of this 3-functor. In Appendix\,\ref{sec:davydov}, we compare our and Davydov's definition of the full center. %in Appendix\,\ref{sec:ends}, we prove a change-of-variable formula for (co-)ends. 

\bigskip
\noindent {\bf Acknowledgement.} Both LK and HZ are supported by the Science, Technology and Innovation Commission of Shenzhen Municipality (Grant Nos. ZDSYS20170303165926217 and JCYJ20170412152620376) and Guangdong Innovative and Entrepreneurial Research Team Program (Grant No. 2016ZT06D348), and by NSFC under Grant No. 11071134. LK is also supported by NSFC under Grant No. 11971219. WY is supported by the NSFC under grant No. 11971463, 11871303, 11871127. HZ is supported by NSFC under Grant No. 11131008.

\section{Elements of tensor categories} \label{sec:fusion-cat}

In this section, we review some basic facts of tensor categories %and some results in \cite{kz1} 
that are important to this work, and set our notations along the way. Throughout the paper, $k$ is an algebraic closed field %of characteristic zero 
and $\bk$ is the symmetric monoidal category of finite-dimensional vector spaces over $k$.

\subsection{Finite monoidal categories and finite modules} \label{sec:finite-cat}
%In this paper, we use the same notation $\otimes$ to denote tensor product functors of monoidal categories and module product functors of module categories. Mimicking the proof of Mac Lance's Coherence theorem (\cite[V.II Theorem 1]{Ma}), it is straightforward to show that every two isomorphisms between two objects in a bimodule category, obtained by compositing module associativities and unit isomorphisms and their inverses possibly tensored with identity morphisms, are equal. Therefore, we omit the parentheses from the expression of tensor products unless confusion is possible.

For a monoidal category $\CC$, we use $\CC^{\rev}$ to denote the same category as $\CC$ but equipped with the reversed tensor product $\otimes^\rev: \CC \times \CC \to \CC$ defined by $a\otimes^\rev b := b\otimes a$. If $\CC$ is rigid, then we use $a^L$ and $a^R$ to denote the left dual and the right dual of an object $a \in \CC$, respectively.

\smallskip

A {\em finite category} over $k$ is a $k$-linear category $\CC$ that is equivalent to the category
of finite-dimensional modules over a finite-dimensional $k$-algebra $A$ (see \cite[Definition 1.8.6]{egno}
for an intrinsic definition). We say that $\CC$ is {\em semisimple} if the defining algebra $A$ is semisimple. 
%A {\em $\bk$-linear functor} $F:\CC\to\CD$ between two $k$-linear categories $\CC$ and $\CD$ is a $k$-linear functor that preserves finite colimits.

%\begin{defn}\label{def:k-linear-monoidal}
A {\em finite monoidal category} over $k$ is a monoidal category $\CC$ such that $\CC$ is a finite category over $k$ and the tensor product $\otimes: \CC \times \CC \to \CC$ is $k$-bilinear  and right exact in each variable. 
A {\em finite multi-tensor category}, or {\em multi-tensor category} for short, is a rigid finite monoidal category. A {\em tensor category} is a multi-tensor category with a simple tensor unit.
A multi-tensor category is {\em indecomposable} if it is neither zero nor the direct sum of two nonzero multi-tensor categories.
A {\em multi-fusion category} is a semisimple multi-tensor category, and a {\em fusion category} is a multi-fusion category with a simple tensor unit.
%A {\em multi-tensor category} is a rigid finite monoidal category. A {\em tensor category} is a multi-tensor category with a simple tensor unit.
%\end{defn}

%\begin{defn} \label{def:finite-module}
Let $\CC$ and $\CD$ be finite monoidal categories. A {\em finite left $\CC$-module} $\CM$ (denoted as ${}_\CC\CM$) is a left $\CC$-module such that $\CM$ is a finite category and the action functor $\odot:\CC\times\CM\to\CM$ is $k$-bilinear and right exact in each variable. 
We say that $\CM$ is {\em indecomposable} if it is neither zero nor the direct sum of two nonzero finite left $\CC$-modules.
The notions of a {\em finite right $\CD$-module} $\CN$ (denoted as $\CN_\CD$) and a {\em finite $\CC$-$\CD$-bimodule} $\CO$ (denoted as ${}_\CC\CO_\CD$) are defined similarly (see \cite[Definition 7.1.7]{egno}).
%\end{defn}

Let $\CC$ and $\CD$ be finite monoidal categories and $\CM, \CN$ be finite $\CC$-$\CD$ bimodules. A {\em $\CC$-$\CD$-bimodule functor} $F: \CM \to \CN$ is a $k$-linear functor equipped with an isomorphism $c\odot F(-)\odot d\simeq F(c\odot-\odot d)$ for $c\in\CC$ and $d\in\CD$ satisfying some natural axioms (see \cite[Definition 7.2.1]{egno}). 
\void{
A {\em $\CC$-$\CD$-bimodule natural transformation} $\gamma: F \to G$ between two bimodule functors $F$ and $G$ is a natural transformation such that the following two diagrams
$$
\xymatrix @R=0.2in{
c\odot F(x) \ar[d]_\simeq \ar[rr]^{\id_c \odot \gamma_x} & & c\odot G(x) \ar[d]^{\simeq} \\
F(c\odot x)  \ar[rr]^{\gamma_{c\odot x}} & & G(c\odot x) 
}
\quad\quad
\xymatrix @R=0.2in{
F(x) \odot d \ar[d]_\simeq  \ar[rr]^{\gamma_x \odot \id_d} & & G(x) \odot d \ar[d]^{\simeq} \\
F(x \odot d) \ar[rr]^{\gamma_{x \odot d}} & & G(x \odot d)  \\
}
$$
are commutative for all $c\in \CC,x\in\CM,d\in\CD$.
} 
We use $\Fun_{\CC|\CD}(\CM, \CN)$ to denote the category of right exact $\CC$-$\CD$-bimodule functors from $\CM$ to $\CN$. 
%Every finite category has a unique (up to equivalence) finite-$\bk$-module structure. 
If $\CC = \bk$ (resp. $\CD = \bk$ or $\CC=\CD=\bk$), we abbreviate $\Fun_{\CC|\CD}(\CM, \CN)$ to $\Fun_{\CD^\rev}(\CM, \CN)$ (resp. $\Fun_{\CC}(\CM, \CN)$ or $\Fun(\CM, \CN)$).  

\begin{rem} \label{rem:adjoint-module-functor}
Let $\CC$ be a rigid monoidal category and $\CM$ and $\CN$ be two left $\CC$-modules. If a left $\CC$-module functor $F:\CM \to \CN$ has a right or left adjoint $G: \CN \to \CM$, then $G$ is automatically a left $\CC$-module functor. 
\end{rem}

An {\em algebra} in a monoidal category $\CC$ is an object $A\in\CC$ equipped with two morphisms $u_A: \one_\CC \to A$ and $m_A: A\otimes A \to A$ in $\CC$ satisfying the unity and associativity properties (see \cite[Definition 7.8.1]{egno}). An {\em algebra homomorphism} $f:A\to B$ is a morphism in $\CC$ such that
\begin{align*}
    u_B = f \circ u_A, \quad m_{B} \circ (f \otimes f) = f \circ m_{A}.
\end{align*}
In what follows, we use $\Alg(\CC)$ to denote the category of algebras in $\CC$ and algebra homomorphisms between them. Note that $\Alg(\CC)$ has an initial object given by the tensor unit $\one_\CC$ of $\CC$.

Let $\CC,\CD$ be monoidal categories and $A \in \Alg(\CC)$, $B\in\Alg(\CD)$. A {\em left $A$-module} in a left $\CC$-module $\CM$ is an object $M\in\CM$ equipped with a morphism $\rho:A \odot M \to M$ in $\CM$ satisfying the usual unity and associativity properties. Similarly, one define the notion of a right $B$-module in a right $\CD$-module $\CN$ and that of an $A$-$B$-bimodule in a $\CC$-$\CD$-bimodule $\CO$. We use $_A\CM$, $\CN_B$, $_A\CO_B$ to denote the category of left $A$-modules in $\CM$, right $B$-modules in $\CN$ and $A$-$B$-bimodules in $\CO$, respectively.

\begin{prop}[{\cite[Proposition 2.3.9]{kz1}}]\label{prop:kz_module_alg}
For a multi-tensor category $\CC$  and a finite left $\CC$-module $\CM$, 
there exists $A \in \Alg(\CC)$ such that $\CM \simeq \CC_A$ as finite left $\CC$-modules. 
%and a monoidal equivalence $_A \CC_A \simeq \Fun_{|\CC}({}_A \CC, {}_A \CC)$ defined by $x \to (x \otimes_A -)$ as multi-fusion categories.
\end{prop}

Given an algebra $A$ in a finite monoidal category $\CC$, we use $x\otimes_A y$ to denote the coequalizer of the parallel morphisms $x\otimes A\otimes y \rightrightarrows x\otimes y$ for a right $A$-module $x$ and a left $A$-module $y$. Note that ${}_A \CC_A$ is a finite monoidal category with tensor product $\otimes_A$ and tensor unit $A$.

\medskip
%\subsection{Tensor product of module categories}

Let $\CC$ be a finite monoidal category, $\CM$ a finite right $\CC$-module and $\CN$ a finite left $\CC$-module. The {\em relative tensor product} of $\CM$ and $\CN$ over $\CC$ is a universal finite category $\CM\boxtimes_\CC\CN$ equipped with a functor $\boxtimes_\CC:\CM\times\CN\to\CM\boxtimes_\CC\CN$, which is $k$-bilinear and right exact in each variable, intertwining the actions of $\CC$. See \cite{tam, eno1, eno3, dd, dss, kz1} for more details. 
Note that, in the special case $\CC = \bk$, $\CM \boxtimes_\bk \CN$ is simply Deligne's tensor product $\CM \boxtimes \CN$. 

\smallskip
 
The following description of relative tensor product is adequate for many purposes.

\begin{thm}[{\cite[Theorem 3.3]{dss}\cite[Theorem 2.2.3]{kz1}}] \label{thm:tensor-product-exits}
   For a multi-tensor category $\CC$, a finite right $\CC$-module $\CM$ and 
    a finite left $\CC$-module $\CN$, the relative tensor product $\CM \boxtimes_{\CC} \CN$ exists. Without loss of generality, suppose that $\CM = {}_{M}\CC$ and $\CN = \CC_{N}$ for $M, N \in \Alg(\CC)$. There are canonical equivalences 
$$\CM \boxtimes_{\CC} \CN \simeq {}_{M}\CC_{N} \simeq \Fun_{\CC}(\CC_{M}, \CC_{N})$$
$$m\boxtimes_\CC n \mapsto m\otimes n \mapsto -\otimes_M m\otimes n.$$
\end{thm}

\void{
\begin{thm}[{\cite[Theorem 2.18]{eno1}}] \label{thm:eno-semisimple}
For a multi-fusion category $\CC$ and two semisimple left $\CC$-modules $\CM,\CN$,  
the category $\Fun_{\CC}(\CM,\CN)$ is a semisimple category and $\Fun_{\CC}(\CM,\CM)$ is a multi-fusion category. 
\end{thm}
}

\subsection{Monoidal modules over a braided monoidal category} \label{sec:relative-over-BTC}

Let $\CC$ be a braided monoidal category with braiding $c_{x,y}: x\otimes y \to y\otimes x$ for $x,y\in \CC$. We use $\overline{\CC}$ to denote the same monoidal category $\CC$ but equipped with the anti-braiding $\bar{c}_{x,y}:=c_{y,x}^{-1}$.
The {\em M\"{u}ger center} of $\CC$, denoted by $\CC'$, is defined to be the full subcategory of $\CC$ consisting of those objects $x$ such that $c_{y,x}\circ c_{x,y}=\id_{x\otimes y}$ for all $y\in\CC$. It is clear that $\one_\CC \in \CC'$. A braided fusion category is called {\em non-degenerate} if its M\"uger center is equivalent to $\bk$.

Recall that the {\em center} (or {\em Drinfeld center} or {\em monoidal center}) of a monoidal category 
$\CC$, denoted by $\FZ(\CC)$, is the category of pairs $(z,\beta_{z,-})$, where $z \in \CC$ and $\beta_{z,-}: z \otimes - \to - \otimes z$ is a natural isomorphism, called a {\em half-braiding} 
%, such that $\beta_{z,x \otimes y} = (\Id_x \otimes \beta_{z,y}) \circ (\beta_{z,x} \otimes \Id_y)$ for all $x,y \in \CC$ and that $\beta_{z,\one_\CC}$ is the composition $z \otimes \one_\CC \simeq z \simeq \one_\CC \otimes z$. A morphism in $\FZ(\CC)$ is a morphism between the first components that preserves the half-braidings 
(see \cite{majid,js}).
The category $\FZ(\CC)$ has a natural structure of a braided monoidal category.
Moreover, $\FZ(\CC)$ can be identified with the category of $\CC$-$\CC$-bimodule functors of $\CC$ (see for example \cite[Theorem 7.16.1]{egno}). 
%If $\CM$ is a fusion category with $\chara k=0$, then $\FZ(\CM)$ is a non-degenerate braided fusion category \cite{muger}. 

\begin{defn}[{\cite{kz1}}] \label{def:monoidal-modules}
Let $\CC$ and $\CD$ be finite braided monoidal categories.
\bnu
\item A {\em monoidal left $\CC$-module} is a finite monoidal category $\CM$ equipped with a right exact $k$-linear braided monoidal functor $\phi_\CM:\overline\CC\to\FZ(\CM)$.
\item A {\em monoidal right $\CD$-module} is a finite monoidal category $\CM$ equipped with a right exact $k$-linear braided monoidal functor $\phi_\CM:\CD\to\FZ(\CM)$.
\item A {\em monoidal $\CC$-$\CD$-bimodule} is a finite monoidal category $\CM$ equipped with a right exact $k$-linear braided monoidal functor $\phi_\CM:\overline\CC\boxtimes \CD\to\FZ(\CM)$.
\enu
Two monoidal $\CC$-$\CD$-bimodules $\CM,\CN$ are {\em equivalent} if there is a $k$-linear monoidal equivalence $\CM\simeq\CN$ such that the composite braided monoidal equivalence $\overline\CC\boxtimes\CD \xrightarrow{\phi_\CM} \FZ(\CM)\simeq\FZ(\CN)$ is isomorphic to $\phi_\CN$.
 
A monoidal $\CC$-$\CD$-bimodule is said to be {\em closed} if $\phi_\CM$ is an equivalence. 
When $\CC$ and $\CD$ are braided fusion categories, a monoidal $\CC$-$\CD$-bimodule $\CM$ is called a {\em multi-fusion (resp. fusion) $\CC$-$\CD$-bimodule} if $\CM$ is a multi-fusion (resp. fusion) category.
\end{defn}

\begin{rem} \label{rem:monoidal-left-right-bi}
A monoidal left $\CC$-module is simply a monoidal $\CC$-$\bk$-bimodule and a monoidal right $\CD$-module is simply a monoidal $\bk$-$\CD$-bimodule. Conversely, a monoidal $\CC$-$\CD$-bimodule is precisely a monoidal right $\overline{\CC}\boxtimes \CD$-module. Since $\overline{\FZ(\CM)} \simeq \FZ(\CM^{\rev})$ canonically as braided monoidal categories, to say that $\CM$ is a monoidal $\CC$-$\CD$-bimodule is equivalent to say that $\CM^\rev$ is a monoidal $\CD$-$\CC$-bimodule.
\end{rem}

\begin{expl}\label{exam:Fun_monoidal_bimodule}
Let $\CC,\CD$ be multi-tensor categories and $\CM$ a finite $\CC$-$\CD$-bimodule. 
Then $\Fun_{\CC|\CD}(\CM,\CM)$ is a monoidal $\FZ(\CC)$-$\FZ(\CD)$-bimodule. 
To see this we may assume without loss of generality that $\CD=\bk$. 
%In particular, $\CM\simeq\Fun_{\bk|\CM}(\CM,\CM)$, $\CM$ is a monoidal right $\FZ(\CM)$-module.
There is a monoidal functor $\phi: \FZ(\CC) \to \Fun_{\CC}(\CM,\CM)$ defined by $c \mapsto c\odot-$, where $\phi_c:=c\odot-$ is equipped with a natural isomorphism for $c'\in\CC$ 
$$\phi_c(c'\odot-) = c\odot(c'\odot-) \xrightarrow{\beta_{c,c'}\odot\id} c'\odot(c\odot-) = c'\odot \phi_c(-),
$$
thus defines a left $\CC$-bimodule functor. 
%The natural isomorphism $R_{d\otimes d'} \simeq R_d\circ R_{d'}$ is given by
%$$R_{d\otimes d'} = -\odot(d\otimes d') \xrightarrow{\id\odot\beta_{d,d'}} (-\odot d')\odot d = R_d\circ R_{d'}.$$
Moreover, $\phi_c$ is equipped with a half-braiding
$$\beta_{\phi_c,F}: \phi_c \circ F 
    = c\odot F(-) \simeq F(c\odot-) 
    = F \circ \phi_c %\label{equ:exam_right_half_brading}
$$
thus defines an object of $\FZ(\Fun_{\CC}(\CM,\CM))$.
Note that $\beta_{\phi_c,\phi_{c'}}$ is given by
$$\phi_c \circ \phi_{c'} = c\odot(c'\odot-) \xrightarrow{\beta_{c',c}^{-1}\odot\id} c'\odot(c\odot-) = \phi_{c'} \circ \phi_c. 
$$
Therefore, $\phi$ is promoted to a braided monoidal functor $\overline{\FZ(\CC)} \to \FZ(\Fun_{\CC}(\CM,\CM))$, as desired. 
%If $\CC$ and $\CD$ are indecomposable multi-fusion categories with $\chara k=0$ and if $\CM$ is a nonzero semisimple $\CC$-$\CD$-bimodule, then $\Fun_{\CC|\CD}(\CM,\CM)$ is a closed multi-fusion $\CC$-$\CD$-bimodule (see \cite[Corollary 7.16.2]{egno}\cite[Proposition 2.4.7]{kz1}).

In the special case $\CM=\CC$, the above construction recovers the canonical braided monoidal equivalence $\overline{\FZ(\CC)} \simeq \FZ(\CC^\rev)$.
\end{expl}

If $\CB$ is a braided multi-tensor category, then for any monoidal right $\CB$-module $\CU$ and any monoidal left $\CB$-module $\CV$, the relative tensor product $\CU\boxtimes_\CB\CV$ has a canonical structure of a finite monoidal category with tensor unit $\one_\CU\boxtimes_\CB \one_\CV$ and tensor product $(x\boxtimes_\CB y)\otimes(x'\boxtimes_\CB y') := (x\otimes x') \boxtimes_\CB (y\otimes y')$ (see \cite{green}).

\subsection{Functoriality of Drinfeld center} \label{sec:kz-thm}

We recall two symmetric monoidal categories introduced in \cite{kz1}. 
\begin{itemize}
\item $\mtc$: an object is an indecomposable multi-tensor category $\CL$ over $k$, a morphism between two objects $\CL$ and $\CM$ is an equivalence class of finite $\CL$-$\CM$-bimodules ${}_\CL\CX_\CM$, and the composition of two morphisms ${}_\CL\CX_\CM$ and ${}_\CM\CY_\CN$ is given by the relative tensor product $\CX\boxtimes_\CM\CY$.  
\item $\btc$: an object is a braided tensor category $\CA$ over $k$, a morphism between two objects $\CA$ and $\CB$ is an equivalence class of monoidal $\CA$-$\CB$-bimodules ${}_\CA\CU_\CB$, and the composition of two morphisms ${}_\CA\CU_\CB$ and ${}_\CB\CV_\CC$ is given by the relative tensor product $\CU\boxtimes_\CB\CV$.
\end{itemize}
The tensor product functors of both categories are Deligne's tensor product $\boxtimes$.

\begin{thm}[{\cite[Theorem 3.1.8]{kz1}}] \label{thm:Z-functor}
The assignment 
$$\CL \mapsto \FZ(\CL), \quad\quad {}_\CL\CX_\CM \mapsto \FZ^{(1)}(\CX):=\Fun_{\CL|\CM}(\CX,\CX)$$
defines a symmetric monoidal functor 
$$\FZ: \mtc \to \btc.$$ 
\end{thm}

We will refer to this functor $\FZ$ as the {\em Drinfeld center functor}.

\begin{rem}
Recall that, by Example \ref{exam:Fun_monoidal_bimodule}, $\Fun_{\CL|\CM}(\CX,\CX)$ is a monoidal $\FZ(\CL)$-$\FZ(\CM)$-bimodule. The functoriality of $\FZ$ follows from the following equivalence: 
\begin{align} \label{eq:isomorphism}
\Fun_{\CL|\CM}(\CX,\CX') \boxtimes_{\FZ(\CM)} \Fun_{\CM|\CN}(\CY,\CY') &\simeq \Fun_{\CL|\CN}(\CX\boxtimes_\CM\CY, \CX'\boxtimes_\CM\CY') \nn
f \boxtimes_{\FZ(\CM)} g &\mapsto f\boxtimes_\CM g, 
\end{align}
and the fact that it is a monoidal equivalence when $\CX'=\CX$, $\CY'=\CY$ \cite[Theorem 3.1.7]{kz1}. 
In particular, in the special case $\CL=\CN=\bk$ and $\CX=\CX'=\CY=\CY'=\CM$, \eqref{eq:isomorphism} reduces to a monoidal equivalence:
$$\CM \boxtimes_{\FZ(\CM)} \CM^\rev \simeq \Fun(\CM,\CM), \quad x\boxtimes_{\FZ(\CM)} y\mapsto x\otimes-\otimes y.$$
\end{rem}
 
\void{
\be \label{eq:Z-functor}
\xymatrix{ \CL \ar@/^12pt/[rr]^{{}_\CL\CX_\CM} \ar@/_12pt/[rr]_{{}_\CL\CX_\CM'} & \Downarrow F & \CM}
\quad \longmapsto \quad
\raisebox{3em}{\xymatrix@R=1.3em{  
  & \Fun_{\CL|\CM}(\CX,\CX) \ar[d]^{F\circ -} & \\ 
  \FZ(\CL) \ar@/^8pt/[ur]^{l\odot -}  \ar@/_8pt/[rd]_{l\odot -} & (\Fun_{\CL|\CM}(\CX,\CX'),F) & \FZ(\CM) \ar@/_8pt/[ul]_{-\odot m} \ar@/^8pt/[dl]^{-\odot m}  \\
  & \Fun_{\CL|\CM}(\CX',\CX') \ar[u]^{-\circ F} 
  }} \ .
\ee
}

In the rest of this subsection, we assume $\chara k=0$.
We recall two symmetric monoidal subcategories $\mfus\subset\mtc$ and $\bfuscl\subset\btc$ introduced in \cite{kz1}. 
\begin{itemize}
\item The category $\mfus$ of indecomposable multi-fusion categories over $k$ with the equivalence classes of nonzero semisimple bimodules as morphisms.
\item The category $\bfuscl$ of non-degenerate braided fusion categories over $k$ with the equivalence classes of closed multi-fusion bimodules as morphisms.
\end{itemize}

\void{
Note that $\mfus$ is well-defined due to Theorem \ref{thm:tensor-product-exits} and Theorem \ref{thm:eno-semisimple}, and $\bfuscl$ is well-defined due to the following two theorems.  

\begin{thm}[{\cite[Corollary 2.7.5]{kz1}}]
If $\CC$ is a braided multi-fusion category and $\CM$ and $\CN$ are multi-fusion right and left $\CC$-modules, respectively, then the monoidal category $\CM \boxtimes_{\CC} \CN$ is a multi-fusion category.
\end{thm}

\begin{thm}[{\cite[Theorem 3.3.6]{kz1}}] \label{thm:closed-bimodule}
Let $\CC$, $\CD$ and $\CE$ be non-degenerate braided fusion categories. Let $\CM$ be a closed multi-fusion $\CC$-$\CD$-bimodule and $\CN$ a closed multi-fusion $\CD$-$\CE$-bimodule. Then $\CM\boxtimes_\CD \CN$ is a closed multi-fusion $\CC$-$\CE$-bimodule. 
\end{thm}
}

Now we are ready to state the main theorem of \cite{kz1}. 

\begin{thm}[{\cite[Theorem 3.3.7]{kz1}}] \label{thm:KZ-2}
The Drinfeld center functor $\FZ: \mtc \to \btc$ restricts to a fully faithful symmetric monoidal functor 
$$\FZ: \mfus \to \bfuscl.$$ 
\end{thm}

\begin{defn} \label{def:non-chiral}
We say that a braided fusion category $\CC$ is {\em non-chiral} if there exists a fusion category $\CM$ such that $\CC \simeq \FZ(\CM)$ as braided fusion categories; if otherwise, way say that $\CC$ is {\em chiral}. 
\end{defn}

We denote the full subcategory of $\bfuscl$ consisting of non-chiral non-degenerate braided fusion categories by $\ncbfuscl$. % (recall Definition \ref{def:non-chiral}). 

\begin{cor} \label{cor:KZ-3}
The Drinfeld center functor $\FZ: \mtc \to \btc$ restricts to a symmetric monoidal equivalence 
$$\mfus \simeq \ncbfuscl.$$ 
\end{cor}

\void{
\begin{rem}
The physical meaning of Theorem \ref{thm:KZ-2} or Corollary \ref{cor:KZ-3} is explained in Section \ref{sec:fusion-anomaly}. We want to remark that a unitary modular tensor category is an $E_2$-algebra; a unitary fusion category is an $E_1$-algebra; the pair $(\CX,x)$ is an $E_0$-algebra \cite{lurie}. 
\end{rem}
}

\section{Centers of an algebra} \label{sec:left-right-full-centers}

\subsection{Internal homs}

Let $\CC$ be a monoidal category and $\CM$ a left $\CC$-module. Given $x, y \in \CM$, the {\em internal hom} $[x,y]_\CC$ in $\CC$, if exists, is defined by the following adjunction: 
$$
\Hom_\CM(c\odot x, y) \simeq \Hom_\CC(c, [x,y]_\CC)
$$
for $c\in\CC$. 
Equivalently, it can be defined as a pair $([x,y]_\CC,\ev_x)$, where $[x,y]_\CC$ is an object in $\CC$ and 
\begin{align}\label{equ:def_ev}
  \ev_x: [x, y]_\CC \odot x \rightarrow y  
\end{align}    
is a morphism in $\CM$, such that $([x, y]_\CC, \ev_x)$ is terminal among all such pairs. That is, for any object $a\in\CC$ and any morphism $f:a\odot x\to y$ in $\CM$, there is a unique morphism $\underline{f}: a\to[x,y]_\CC$ in $\CC$
rendering the following diagram
\begin{align}\label{eq:underline-f}
\raisebox{2em}{ \xymatrix @C=1.5em@R=1.5em{
     & [x,y]_\CC\odot x \ar[rd]^{\ev_x} & \\
 a\odot x \ar[rr]^f   \ar@{->}[ur]^{\underline{f}\odot \id_x}  &  & y   
}}   
\end{align}  
commutative. In other words, $([x, y]_\CC, \ev_x)$ is the terminal object in the comma category $(\CC \odot x\downarrow y)$. 
%\footnote{Here the role of \mref{equ:def_ev} on the side of the arrows is to indicate the morphism represented by the arrow is induced by \mref{equ:def_ev}. In the following, the same method will be used to explain the meaning of arrows in commutative diagrams, i.e., the arrow is induced by (repeatedly) applying the indexed morphism.}.
For simplicity, we will sometimes abbreviate $[x,y]_{\CC}$ to $[x,y]$.

\medskip
We say that $\CM$ is {\em enriched} in $\CC$, if $[x,y]_\CC$ exists for every pair of objects $x, y \in \CM$. 
If $\CC$ is a finite monoidal category, then every finite left $\CC$-module $\CM$ is {\em enriched in $\CC$}
(see \cite[Lemma 2.3.7]{kz1}). In this case, $[-,-]_\CC: \CM^\op \times \CM \to \CC$ is a $k$-bilinear functor. Moreover, if $\CM = \CC_A$, where $\CC$ is a multi-tensor category and $A$ is an algebra in $\CC$, then 
\begin{align*}
  [x,y]_\CC \simeq (x\otimes_A y^R)^L
\end{align*}
(see for example \cite{ostrik}\cite[Lemma 2.1.6]{kz1}). It is known that $[x,x]_\CC$ is an algebra in $\CC$ and $[x,y]_\CC$ is a $[y,y]_\CC$-$[x,x]_\CC$ bimodule (see for example \cite{ostrik}, \cite[Section\, 7.9]{egno}).

\medskip
For a multi-fusion category $\CC$, if $\CX$ is an indecomposable semisimple left $\CC$-module then we have an equivalence of left $\CC$-modules $\CX\simeq\CC_{[x,x]}$, $y\mapsto [x,y]$ for every nonzero $x\in\CC$ (see for example \cite[Corollary 3.5]{kz4}).

\subsection{Definitions of left and right centers} 

Let $\CC$ and $\CD$ be finite braided monoidal categories. 
Note that $\Alg(\CC)$ is a monoidal category such that, for $X,Y\in\Alg(\CC)$, the multiplication of the algebra $X\otimes Y$ is defined by
\begin{align*} \label{eq:braiding-convention-1}
    X \otimes Y \otimes X \otimes Y \xrightarrow{\id \otimes c_{Y,X} \otimes \id}
    X \otimes X \otimes Y \otimes Y \xrightarrow{m_X \otimes m_Y} X \otimes Y.
\end{align*}
Note that $\Alg(\overline{\CC}) \simeq \Alg(\CC)^\rev$ canonically as monoidal categories.

If $\CM$ is a monoidal right $\CD$-module, then $\Alg(\CM)$ is a right $\Alg(\CD)$-module in the following way. For $X\in\Alg(\CD)$ and $A\in\Alg(\CM)$, the unit and the multiplication of $A \odot X$ are given respectively by
\begin{align*}
    \one_\CM \simeq \one_\CM \odot \one_\CD \xrightarrow{u_A \odot u_X} A \odot X,\quad 
    (A \odot X)\otimes (A \odot X) \simeq (A \otimes A) \odot (X \otimes X)\xrightarrow{m_A \odot m_X} A \odot X.
\end{align*}
Similarly, if $\CM$ is a monoidal left $\CC$-module, then $\Alg(\CM)$ is a left $\Alg(\CC)$-module; if $\CM$ is a monoidal $\CC$-$\CD$-bimodule then $\Alg(\CM)$ is an $\Alg(\CC)$-$\Alg(\CD)$-bimodule (see Remark \ref{rem:monoidal-left-right-bi}).

\begin{defn}
Let $\CM$ be a monoidal right $\CD$-module and $A\in\Alg(\CM)$, $X\in\Alg(\CD)$. A {\em unital $X$-action} on $A$ is a morphism $f: A\odot X \to A$ in $\Alg(\CM)$ such that the composition $A \simeq A \odot \one_\CD \xrightarrow{\id_A \odot u_X} A \odot X \xrightarrow{f} A$ coincides with $\id_A$,
\end{defn}

\begin{defn} \label{def:left-right-full-centers}
Let $\CM$ be a monoidal right $\CD$-module and $A\in\Alg(\CM)$. The {\em right center} of $A$ in $\CD$ is a pair $(Z_\CD(A),m)$, where $Z_\CD(A)\in\Alg(\CD)$ and  $m: A \odot Z_\CD(A) \to A$ is a unital $Z_\CD(A)$-action on $A$, such that it is terminal among all such pairs.
In the special case where $\CD=\FZ(\CM)$, $Z_\CD(A)$ is also denoted by $Z(A)$, called the {\em full center} of $A$.

For $A \in \Alg(\CN)$ where $\CN$ is a monoidal left $\CC$-module, the {\em left center} of $A$ in $\CC$ is defined to be the right center of $A$ in $\overline{\CC}$ by regarding $\CN$ as a monoidal right $\overline{\CC}$-module. 
\end{defn}

\begin{rem}
The right center $Z_\CD(A)$ is equipped with a canonical algebra homomorphism $\phi_\CM(Z_\CD(A)) \to A$ given by the composition $\phi_\CM(Z_\CD(A)) \simeq \one_\CM\odot Z_\CD(A) \xrightarrow{u_A\odot\id_{Z_\CD(A)}} A\odot Z_\CD(A) \xrightarrow{m} A$.
\end{rem}

\begin{rem}\label{rem:basic_fact_center}
The universal properties of the left and right centers of $A$ can be illustrated by the following diagrams, respectively: 
\be \label{diag:universal-property}
\raisebox{4em}{
\xymatrix@R=0.5em@C=0.8em{
&  & Z_\CC(A) \odot A \ar@/^1.5pc/[rrdddd]^{m} &  & \\
& & &  & \\
& & X \odot A \ar[uu]^{\exists !\, \underline{f} \odot \id_A} \ar[rrdd]^f &  & \\
& &  & & \\
\one_\CC\odot A \ar[uurr]^{u_X \odot \id_A} \ar@/^1.5pc/[uuuurr]^{u_{Z_\CC(A)} \odot \id_A}  \ar[rrrr]^{\sim} & & & & A\, ,
}}
\quad\quad
\raisebox{4em}{
\xymatrix@R=0.5em@C=0.8em{
&  & A \odot Z_\CD(A)  \ar@/^1.5pc/[rrdddd]^{m} &  & \\
& & &  & \\
& & A \odot Y \ar[uu]^{\exists !\,  \id_A \odot \underline{g}} \ar[rrdd]^g &  & \\
& &  & & \\
A\odot\one_\CD \ar[uurr]^{\id_A \odot u_Y} \ar@/^1.5pc/[uuuurr]^{ \id_A\odot u_{Z_\CD(A)}}  \ar[rrrr]^{\sim} & & & & A\, ,
}}
\ee
where $m, f, g$ are all algebra homomorphisms. 
\end{rem}

%We introduce the following notions that will be used in Section \ref{sec:center-functor}. 
%\begin{defn} \label{def:algebraic-module} For $A\in \Alg(\CM)$, $X\in\Alg(\CC)$ and $Y\in\Alg(\CD)$, $A$ is called an algebraic left $X$-module in $\Alg(\CM)$ if $A$ is equipped with a unital $X$-action; $A$ is called an algebraic right $Y$-module in $\Alg(\CM)$ if $A$ is equipped with a unital $Y$-action. \end{defn}

\begin{expl}
When $\CM=\bk$, an algebra $A$ in $\CM$ is just an ordinary finite-dimensional $k$-algebra. The usual center $Z(A)$ of $A$ is defined as the subalgebra $Z(A)=\{ z\in A \mid az=za, \forall a\in A\}$. Let $m: A \otimes_k Z(A) \to A$ be the multiplication map, which clearly defines a unital $Z(A)$-action on $A$. It is easy to check that the pair $(Z(A), m)$ is the right (and left because $\bk$ is symmetric) center of $A$ defined above. 
\end{expl}

\begin{rem}
In the special case $\CM=\CC$, the notion of left/right center of an algebra $A$ in $\CC$ was introduced by Ostrik \cite[Definition\,15]{ostrik}. In fact, $\CC$ is a monoidal $\CC$-$\CC$-bimodule defined by the evident braided monoidal functor $\overline{\CC} \boxtimes \CC \to \FZ(\CC)$. For an algebra $A\in\Alg(\CC)$, the notion of left (resp. right) center of $A$ in $\CC$ defined here coincides with that of right (resp. left) center of $A$ defined by Ostrik . 
\end{rem}

\begin{rem}
The full center $Z(A)$ coincides with the one introduced by Davydov \cite{davydov}. We will explain this in Appendix \ref{sec:davydov}. 
\end{rem}

\subsection{Centers as internal homs}

%Comparing Diagrams (\ref{eq:underline-f}) and (\ref{diag:universal-property}), we see that the only difference between the universal properties of the internal hom and that of the left and right centers is that $m,f,g$ in (\ref{diag:universal-property}) are required to satisfy an additional condition: making the lower and outer triangles commutative (i.e. both actions are unital). Note that the $A$-$A$-bimodule category ${_A}\CM_A$ is a finite monoidal category for every $A \in \Alg(\CM)$ since $\CM$ is finite. We will show that the additional condition in the definition of left/right center can be absorbed by the requirement that $m,f,g$ are morphisms in $\Alg({_A}\CM_A)$. Then $Z_\CC(A)$ and $Z_\CD(A)$ are indeed internal homs. 

%\medskip

Let $\CC$ and $\CD$ be finite braided monoidal categories. 
If $\CM$ is a monoidal right $\CD$-module then, for any algebra $A\in\Alg(\CM)$, ${}_A \CM_A$ is a monoidal right $\CD$-module defined by the braided monoidal functor $\CD\to\FZ({}_A \CM_A)$, $d\mapsto A\odot d$. It follows that $\Alg({_A}\CM_A)$ is a right $\Alg(\CD)$-module.
Similarly, if $\CM$ is a monoidal left $\CC$-module, then $\Alg({_A}\CM_A)$ is a left $\Alg(\CC)$-module for any algebra $A\in\Alg(\CM)$.

\begin{lem}\label{lem:alg_hom_B_to_A}
Let $A$ be an algebra in a finite monoidal category $\CM$ and $B$ an algebra in ${}_A \CM_A$. Giving an algebra homomorphism $h:B\to A$ in ${}_A \CM_A$ is equivalent to giving an algebra homomorphism $h:B\to A$ in $\CM$ such that the composition $A \xrightarrow{u_B} B \xrightarrow{h} A$ is $\id_A$.
\end{lem}

\begin{proof}
Let $h:B\to A$ be an algebra homomorphism in ${}_A \CM_A$. Then the right square of the following diagram is commutative:
    \begin{align}\label{diag:alg_hom_B_to_A_diag_1}
        \xymatrix @R=0.2in{
            B \otimes B \ar@{->>}[r] \ar[d]^{h \otimes h}& B \otimes_A B \ar[r]^-{m_B} \ar[d]^{h \otimes_A h} &B \ar[d]^{h}\\
            A \otimes A \ar@{->>}[r] & A \otimes_A A \ar[r]^-\sim &A .
        }
    \end{align}
Since $h$ is an $A$-$A$-bimodule map, the left square is also commutative. The commutativity of the outer square then states that $h$ defines an algebra homomorphism in $\CM$.
Since $A$ is an initial object of $\Alg({}_A \CM_A)$, we have $h\circ u_B=\id_A$.

Conversely, let $h:B\to A$ be an algebra homomorphism in $\CM$ such that $h\circ u_B=\id_A$. Then the following diagram commutes:
    \begin{align*}
        \xymatrix{
            A \otimes B \ar[r]^{u_B \otimes \id_B} \ar[rd]_{\id_A \otimes h} 
            & B \otimes B \ar[r] \ar[d]^{h \otimes h} & B \ar[d]^{h} & B \otimes B \ar[l] \ar[d]^{h \otimes h}& 
            B \otimes A \ar[l]_{\id_B \otimes u_B} \ar[ld]^{h \otimes \id_A}\\
            & A \otimes A \ar[r] & A & \ar[l] A \otimes A .
        }
    \end{align*}
Thus $h$ is an $A$-$A$-bimodule map. Since the outer and the left squares of Diagram \eqref{diag:alg_hom_B_to_A_diag_1} are commutative, so is the right one. That is, $h$ defines an algebra homomorphism in ${}_A \CM_A$.
\end{proof}

\begin{thm} \label{thm:left-center-definition-2}
Let $\CM$ be a monoidal right $\CD$-module and $A \in \Alg(\CM)$. Then the right center $Z_\CD(A)$ is the internal hom $[\one_{{_A}\CM_A},\one_{{_A}\CM_A}]_{\Alg(\CD)^\rev}$, where $\one_{{_A}\CM_A}=A$ is the trivial algebra in ${_A}\CM_A$. 
\end{thm}

\begin{proof}
According to Lemma \ref{lem:alg_hom_B_to_A}, giving an algebra homomorphism $A\odot X \to A$ in ${_A}\CM_A$ is equivalent to giving a unital $X$-action on $A$ for $X\in\Alg(\CD)$. That is, the internal hom $[A,A]_{\Alg(\CD)^\rev}$ and the right center $Z_\CD(A)$ share the same universal property.
\end{proof}

\begin{rem}
Similarly, if $\CM$ is a monoidal left $\CC$-module, then the left center $Z_\CC(A)$ of $A \in \Alg(\CM)$ is the internal hom $[\one_{{_A}\CM_A},\one_{{_A}\CM_A}]_{\Alg(\CC)}$.
%Let $A \in \Alg(\CM)$, where $\CM$ is a monoidal left $\CC$-module. By Remark \ref{rem:monoidal-left-right-bi} and Theorem \ref{thm:left-center-definition-2}, $\Alg({_A}\CM_A)$ is a left $\Alg(\CD)$ and the left center $Z_\CC(A)$ is the internal hom $[\one_{{_A}\CM_A},\one_{{_A}\CM_A}]_{\Alg(\CC)}$.
\end{rem}

\begin{cor} \label{cor:morita-invariant}
Morita equivalent algebras share the same left and right centers. In other words, left and right centers are Morita invariants. 
\end{cor}

Given a braided monoidal category $\CB$, we use $\CAlg(\CB)$ to denote the category of commutative algebras in $\CB$. The following result is a special case of a general fact proved by Lurie \cite{lurie}. For the reader's convenience, we briefly unravel the proof. 

\begin{lem} \label{lem:CAlg-AlgAlg}
$\CAlg(\CB)=\Alg(\Alg(\CB))$, where ``$=$'' means canonically isomorphic as categories. 
\end{lem}

\begin{proof}
Suppose that $B\in\Alg(\Alg(\CB))$. We use $u:\one_\CB\to B$ and $m: B \otimes B \to B$ to denote the unit and multiplication of $B$ as an algebra in $\Alg(\CB)$, and use $u_B:\one_\CB\to B$ and $m_B: B \otimes B \to B$ to denote the unit and multiplication of $B$ as an algebra in $\CB$.
Since $u$ is an algebra homomorphism, we have $u=u_B$. Since $m$ is an algebra homomorphism, we have
\be \label{eq:h-iota_A}
m_B \circ (m\otimes m) = m \circ (m_B\otimes m_B) \circ (\id_B\otimes c_{B,B}\otimes\id_B).
\ee
Composing both sides of \eqref{eq:h-iota_A} with $\id_B\otimes u\otimes u\otimes\id_B$ from the right, we obtain $m_B=m$. 
Composing both sides of \eqref{eq:h-iota_A} with $u\otimes \id_B\otimes\id_B\otimes u$ from the right, we obtain $m_B=m\circ c_{B,B}$.
It follows that $B$ is a commutative algebra in $\CB$, i.e. $m_B=m_B\circ c_{B,B}$.

Conversely, suppose that $B\in\CAlg(\CB)$. We have 
$$m_B \circ (m_B\otimes m_B) = m_B \circ (m_B\otimes m_B) \circ (\id_B\otimes c_{B,B}\otimes\id_B)$$
which amounts to that $m_B: B \otimes B \to B$ is an algebra homomorphism. Thus the triple $(B,u_B,m_B)$ define an object of $\Alg(\Alg(\CB))$.
The above two constructions are clearly inverse to each other.  
\end{proof}

\begin{cor} \label{cor:ZA-commutative}
Left and right centers are commutative algebras. 
%If $Z_\CD(A)$ (resp. $Z_\CC(A)$) exists, then it is a commutative algebra in $\CD$ (resp. $\CC$). 
\end{cor}
\begin{proof}
In the situation of Theorem \ref{thm:left-center-definition-2}, the right center $Z_\CD(A) \simeq [\one_{{_A}\CM_A},\one_{{_A}\CM_A}]_{\Alg(\CD)^\rev}$ belongs to $\Alg(\Alg(\CD))=\CAlg(\CD)$ hence is commutative. The same is true for left center.
\end{proof}

Let $\CM$ be a monoidal right $\CD$-bimodule and $A \in \Alg(\CM)$, $Y \in \Alg(\Alg(\CD))=\CAlg(\CD)$. By Theorem \ref{thm:left-center-definition-2}, endowing $A$ with the structure of a right $Y$-module is equivalent to giving an algebra homomorphism $Y\to Z_\CD(A)$, i.e. giving a unital $Y$-action $\rho: A \odot Y \to A$.

\begin{defn} \label{def:closed-AB-bimodule}
Let $\CM$ be a closed monoidal $\CC$-$\CD$-bimodule, i.e. $\phi_\CM: \overline{\CC} \boxtimes \CD \to \FZ(\CM)$ is a braided monoidal equivalence. For $X\in\CAlg(\CC)$ and $Y\in\CAlg(\CD)$, we say that an $X$-$Y$-bimodule $A$ in $\Alg(\CM)$ is {\em closed} if the associated algebra homomorphism $\phi_\CM(X\boxtimes Y) \to Z(A)$ is an isomorphism. 
\end{defn}

\subsection{Computing centers}

Let $\CC$ and $\CD$ be finite braided monoidal categories. 

\begin{lem}\label{lem:internal_hom_alg_hom}
Let $\CM$ be a monoidal right $\CD$-bimodule. Then $\ev: \one_\CM \odot [\one_\CM, \one_\CM]_\CD \to \one_\CM$ is an algebra homomorphism.  
\end{lem}

\begin{proof}
We abbreviate $\one_\CM$ to $\one$ and $[\one_\CM, \one_\CM]_\CD$ to $[\one,\one]$ for simplicity.
It is clear that $\ev$ preserves unit.
By definition, the multiplication $m_{[\one,\one]}$ of the algebra $[\one,\one]$ is the unique morphism rendering the following diagram commutative:
$$\xymatrix@R=1.5em{
  \one\odot([\one,\one]\otimes[\one,\one]) \ar[rr]^-{\id_\one\odot m_{[\one,\one]}} \ar[d]^\sim & & \one\odot[\one,\one] \ar[d]^\ev \\
  (\one\odot[\one,\one])\odot[\one,\one] \ar[r]^-{\ev\odot\id_{[\one,\one]}} & \one\odot[\one,\one] \ar[r]^-\ev & \one .
}
$$
Identifying $\one\odot d=\one\otimes\phi_\CM(d)$ with $\phi_\CM(d)$ for $d\in\CD$, we see that the above diagram is equivalent to the following one:
$$\xymatrix@R=1.5em{
  (\one\odot[\one,\one])\otimes(\one\odot[\one,\one]) \ar[d]^{\ev\otimes\ev} \ar[r]^-\sim & (\one\otimes\one)\odot([\one,\one]\otimes[\one,\one]) \ar[r]^-{m_\one\odot m_{[\one,\one]}} & \one\odot[\one,\one] \ar[d]^\ev \\
  \one\otimes\one \ar[rr]^-{m_\one} & & \one .
}
$$
Therefore, $m_{[\one,\one]}$ is the unique morphism rendering $\ev$ an algebra homomorphism.    
\end{proof}

\begin{prop}\label{lem:internal_hom_to_alg_internal_hom}
Let $\CM$ be a monoidal right $\CD$-bimodule. 
%Since $\CM={_{\one_\CM}}\CM_{\one_\CM}$ as monoidal categories, 
We have a canonical algebra isomorphism
$$
[\one_\CM, \one_\CM]_{\Alg(\CD)^\rev} \simeq [\one_\CM, \one_\CM]_\CD,
$$
where $\one_\CM$ is regarded as an object of $\Alg(\CM)$ on the left hand side, an object of $\CM$ on the right hand side.
\end{prop}

\begin{proof}
We abbreviate $\one_\CM$ to $\one$ and $[\one_\CM, \one_\CM]_\CD$ to $[\one,\one]$ for simplicity.
By definition, we have an adjunction for $Y \in \CD$:  
$$%\be \label{eq:1M1MC}
\Hom_{\CM}(\one \odot Y, \one) \simeq \Hom_{\CD}(Y, [\one, \one]). 
$$%\ee
We need to show that it restricts to an adjunction for $Y \in \Alg(\CD)$: 
\begin{align*}
\Hom_{\Alg(\CM)}(\one \odot Y, \one) &\simeq \Hom_{\Alg(\CD)}(Y, [\one, \one]).        
\end{align*}
According to Lemma \ref{lem:internal_hom_alg_hom}, the mate $\ev\circ(\id_{\one}\odot g) : \one \odot Y \to \one$ of an algebra homomorphism $g:Y \to [\one, \one]$ is an algebra homomorphism. It remains to show that the mate of an algebra homomorphism $f: \one\odot Y \to \one$ is an algebra homomorphism. That is, the unique morphism $\underline{f}:Y \to [\one, \one]$ rendering the diagram
    $$ %\label{diag:f-ev}
        \xymatrix@R=2em @C=2em{
             & \one \odot [\one, \one]  \ar[rd]^-{\ev} & \\
             \one\odot Y \ar[ur]^-{\id_{\one}\odot \underline{f}} \ar[rr]^{f} & & \one
        }
    $$
commutative is an algebra homomorphism.

By the universal property of the internal hom $[\one, \one]$, the composition $\one_\CD \xrightarrow{u_Y} Y \xrightarrow{\underline{f}} [\one, \one]$ agrees with $u_{[\one, \one]}$. This shows that $\underline{f}$ preserves unit.
Consider the following diagram:
\begin{align*}
\small
    \xymatrix {%@R = 0.2in @C=0.3in{
        (\one \odot Y) \otimes (\one \odot Y) \ar[r]^-\sim \ar[d]^{(\id_{\one} \odot \underline{f})\otimes (\id_{\one}\odot \underline{f})}& (\one \otimes \one)\odot (Y \otimes Y) \ar[r]^-{m_{\one}\odot m_Y} \ar[d]^{\id_{\one \otimes \one}\odot (\underline{f} \otimes \underline{f})} & \one \odot Y \ar[d]^{\id_{\one}\odot \underline{f}}\\
    (\one \odot [\one, \one]) \otimes (\one\odot [\one, \one]) \ar[r]^-\sim \ar[d]^{\ev\otimes\ev} & (\one \otimes \one) \odot ([\one, \one] \otimes [\one, \one]) \ar[r]^-{m_{\one}\odot m_{[\one, \one]}} & \one\odot [\one, \one] \ar[d]^{\ev}\\
    \one\otimes\one \ar[rr]^-{m_\one} & & \one .
}
\end{align*}
The left-top square is commutative because $\underline{f}$, as a morphism in $\CD$, preserves half-braiding. The bottom and the outer squares are commutative because $\ev$ and $f$ are algebra homomorphisms. Then by the universal property of the internal hom $[\one, \one]$, we read off from the right-top square an equality $\underline{f}\circ m_Y = m_{[\one,\one]}\circ(\underline{f}\otimes\underline{f})$. Namely, $\underline{f}$ is an algebra homomorphism. 
\void{
Consider the following diagram:
\begin{align*}
\small
    \xymatrix {%@R = 0.2in @C=0.3in{
        (\one \odot Y) \otimes (\one \odot Y) \ar[r]^-\sim \ar[d]^{(\id_{\one} \odot \underline{f})\otimes (\id_{\one}\odot \underline{f})}& (\one \otimes \one)\odot (Y \otimes Y) \ar[r]^-{m_{\one}\odot m_Y} \ar[d]^{\id_{\one \otimes \one}\odot (\underline{f} \otimes \underline{f})} & \one \odot Y \ar[d]^{\id_{\one}\odot \underline{f}}\\
    (\one \odot [\one, \one]) \otimes (\one\odot [\one, \one]) \ar[r]^-\sim & (\one \otimes \one) \odot ([\one, \one] \otimes [\one, \one]) \ar[d]^{m_{\one}\odot m_{[\one, \one]}} & \one\odot [\one, \one] \ar[d]^{\ev}\\
    & \one\odot [\one, \one] \ar[r]^{\ev} & \one .
}
\end{align*}
The left square is commutative because $\underline{f}$, as a morphism in $\CD$, respects half-braiding. The outer two paths agree because $\ev$ and $f$ are algebra homomorphisms. Therefore, the right square is commutative. This implies that $\underline{f}$ is an algebra homomorphism by the universal property of $[\one, \one]$.
}
\end{proof}

\begin{cor} \label{cor:center-A}
Let $\CM$ be a monoidal $\CC$-$\CD$-bimodule. We have the following algebra isomorphisms for $A\in\Alg(\CM)$:
\begin{enumerate}
\item[$(1)$] $Z_\CC(A) \simeq [\one_{{_A}\CM_A}, \one_{{_A}\CM_A}]_\CC$ and $Z_\CD(A) \simeq [\one_{{_A}\CM_A}, \one_{{_A}\CM_A}]_\CD$;
\item[$(2)$] $Z_\CC(\one_\CM)\simeq[\one_\CM,\one_\CM]_\CC$ and $Z_\CD(\one_\CM)\simeq[\one_\CM,\one_\CM]_\CD$;
\item[$(3)$] $Z_\CC(A) \simeq Z_\CC(\one_{{_A}\CM_A})$ and $Z_\CD(A) \simeq Z_\CD(\one_{{_A}\CM_A})$.
%\item[$(4)$] $Z(A)\simeq[\one_{{_A}\CM_A},\one_{{_A}\CM_A}]_{\FZ(\CM)}$ and $Z(\one_\CM)\simeq[\one_\CM,\one_\CM]_{\FZ(\CM)}$.
\end{enumerate}
In particular, all these centers exist. 
\end{cor}

\begin{proof}
(1) Combine Theorem \ref{thm:left-center-definition-2} and Proposition \ref{lem:internal_hom_to_alg_internal_hom}. (2) is a special case of (1). (3) is a consequence of (1) and (2).
\end{proof}

\begin{expl}
In the special case $\CC=\CD=\CM=\bk$, Corollary \ref{cor:center-A}(1) states that $Z(A)\simeq[\one_{{_A}\bk_A},\one_{{_A}\bk_A}]_\bk$. The right hand side is exactly $\hom_{A|A}(A,A)$, the algebra of $A$-$A$-bimodule maps of $A$.
\end{expl}

\begin{cor} \label{cor:center-sum}
Let $\CM$ be a monoidal $\CC$-$\CD$-bimodule. Then $Z_\CC(A\oplus B)\simeq Z_\CC(A)\oplus Z_\CC(B)$ and $Z_\CD(A\oplus B)\simeq Z_\CD(A)\oplus Z_\CD(B)$ for $A,B\in\Alg(\CM)$.
\end{cor}

\begin{cor}
Let $\CM,\CN$ be multi-tensor categories. Then $Z(A\boxtimes B) \simeq Z(A)\boxtimes Z(B)$ for $A\in\Alg(\CM)$, $B\in\Alg(\CN)$ under the identification $\FZ(\CM\boxtimes\CN)=\FZ(\CM)\boxtimes\FZ(\CN)$.
\end{cor}

The following lemma gives special cases of \cite[Proposition\,3.5]{dkr}. 
\begin{lem} \label{thm:center-end}
Let $A$ be an algebra in a multi-tensor category $\CC$.

$(1)$ The end $\int_{x\in\CC_A}[x,x\otimes_A w]_\CC$ is equipped with a canonical half-braiding hence defines an object of $\FZ(\CC)$ for $w\in{}_A\CC_A$.

$(2)$ The functor $w \mapsto \int_{x\in\CC_A}[x,x\otimes_A w]_\CC$ is right adjoint to the functor $\FZ(\CC)\to {}_A\CC_A$, $b\mapsto b\otimes A$.

$(3)$ We have $[\one_{{_A}\CC_A},w]_{\FZ(\CC)} \simeq \int_{x\in\CC_A}[x,x\otimes_A w]_\CC$ for $w\in{}_A\CC_A$.

$(4)$ We have $Z(A) \simeq \int_{x\in\CC_A}[x,x]_\CC$.
\end{lem}

\begin{proof}
(1) We have $\int_{x\in\CC_A}[x,x\otimes_A w]_\CC\otimes a
\simeq \int_{x\in\CC_A}[a^R\otimes x,x\otimes_A w]_\CC
\simeq \int_{x\in\CC_A}[x,a\otimes x\otimes_A w]_\CC
\simeq a\otimes\int_{x\in\CC_A}[x,x\otimes_A w]_\CC
$ for $a\in\CC$, where the second isomorphism is due to the fact that the functor $a^R\otimes-:\CC_A\to\CC_A$ is right adjoint to $a\otimes-$.

(2) The unit map $u_b: b\simeq b\otimes\one_{\FZ(\CC)} \xrightarrow{\id_b\otimes u} b\otimes \int_{x\in\CC_A}[x,x]_\CC \simeq \int_{x\in\CC_A}[x,b\otimes x]_\CC$, where $u$ is induced by the canonical family $\one_\CC\to[x,x]_\CC$, and the counit map $v_w: \int_{x\in\CC_A}[x,x\otimes_A w]_\CC\otimes A \to [A,A\otimes_A w]_\CC\otimes A \xrightarrow{\ev} w$ exhibit the adjunction.

(3) is a consequence of (2).

(4) is a consequence of (3) and Corollary \ref{cor:center-A}(1).
\end{proof}

\begin{prop} \label{prop:coequ}
Let $\CC$ be a tensor category. Consider the following morphisms for $x\in\CC$
\begin{align*}
& \lambda_x: Z(\one_\CC)\otimes x \to \one_\CC\otimes x \simeq x, \\
& \rho_x: Z(\one_\CC)\otimes x \simeq x\otimes Z(\one_\CC) \to x\otimes \one_\CC \simeq x.
\end{align*}
$(1)$ We have $\lambda_x=\rho_x$ if and only if $x$ is a direct sum of $\one_\CC$.
$(2)$ The coequalizer of $\lambda_x$ and $\rho_x$ is $\Hom_\CC(x,\one_\CC)^\vee\otimes\one_\CC$.
\end{prop}

\begin{proof}
(1) We may identify $Z(\one_\CC)$ with $\int_{a\in\CC} a\otimes a^L$ by Lemma \ref{thm:center-end}(4). Let $h_b:Z(\one_\CC) \to b\otimes b^L$ be the canonical morphism.
Unwinding the proof of Lemma \ref{thm:center-end}, we see that the canonical algebra homomorphism $Z(\one_\CC)\to\one_\CC$ is given by $h_{\one_\CC}$. Moreover, $\rho_x$ is given by the composition $Z(\one_\CC) \otimes x \xrightarrow{h_x\otimes\id_x} x\otimes x^L\otimes x \xrightarrow{\id_x\otimes v_x} x\otimes\one_\CC \simeq x$. 
In summary, the morphism $Z(\one_\CC)\to x\otimes x^L$ induced by $\lambda_x$ is the composition of $h_{\one_\CC}$ with $u_x:\one_\CC\to x\otimes x^L$ while that induced by $\rho_x$ coincides with $h_x$.
If $\lambda_x=\rho_x$ then $h_x$ factors through $h_{\one_\CC}$. Since $\CC$ is a tensor category, the tensor product of $\CC$ is exact in each variable and we have $a\otimes b\not\simeq0$ for simple objects $a,b\in\CC$. Therefore, $\lambda_x=\rho_x$ implies that the canonical morphism $\int_{a\in\CC} a\boxtimes a^L \to x\boxtimes x^L$ in $\CC\boxtimes\CC$ factors through $\int_{a\in\CC} a\boxtimes a^L \to \one_\CC\boxtimes \one_\CC^L$. This is possible only if $x$ is a direct sum of $\one_\CC$ because $\one_\CC$ is a simple object of $\CC$.

(2) According to (1), the coequalizer of $\lambda_x$ and $\rho_x$ is the maximal quotient of $x$ that is a direct sum of $\one_\CC$, which is exactly $\Hom_\CC(x,\one_\CC)^\vee\otimes\one_\CC$.
\end{proof}

\section{Pointed Drinfeld center functor} \label{sec:Drinfeld-full-center-functor}

In this section, we show that the functoriality of Drinfeld center \cite{kz1} and that of full center \cite{dkr} can be combined into a new center functor involving both Drinfeld center and full center (see Theorem \ref{thm:Df-center}). Moreover, this new center functor restricts to a symmetric monoidal equivalence (see Theorem \ref{thm:ff}). These results generalize many earlier results in the literature.

\subsection{Exact algebras}

Exact algebras are a class of algebras in multi-tensor categories introduced by Etingof and Ostrik \cite{eo} in analogy with semisimple algebras in multi-fusion categories.

\begin{defn}[\cite{eo}]
Let $\CC$ be a multi-tensor category. A finite left $\CC$-module $\CM$ is {\em exact} if for any projective $P\in\CC$ and any object $x\in\CM$ the object $P\odot x$ is projective.
An algebra $A$ in $\CC$ is {\em exact} if the left $\CC$-module $\CC_A$ is exact.
\end{defn}

\begin{expl}
%Suppose that $\chara k=0$.
An algebra $A$ in a multi-fusion category $\CC$ is exact if and only if $A$ is semisimple in the sense that $\CC_A$ is semisimple. Indeed, $A$ is exact $\Leftrightarrow$ $P\otimes x$ is projective for any $P\in\CC$ and $x\in\CC_A$ $\Leftrightarrow$ any $x\in\CC_A$ is projective $\Leftrightarrow$ $\CC_A$ is semisimple. 
%Cf. \cite[Example 3.3(iii)]{eo}.
Moreover, if $\chara k=0$ then $A$ is semisimple if and only if $A$ is separable (see for example \cite[
Theorem 6.10]{kz4}).
\end{expl}

\void{
\begin{rem}
According to \cite[Proposition 3.16 and Lemma 3.21]{eo}, a finite left $\CC$-module $\CM$ is exact if and only if $\Fun_\CC(\CM,\CM)$ is rigid. Therefore, an algebra $A$ in $\CC$ is exact if and only if ${}_A\CC_A$ is a multi-tensor category.
\end{rem}
}

The following proposition can be derived easily from the results of  \cite{eo}. For the reader's convenience we sketch a proof.

\begin{prop}
Let $\CC$ be a multi-tensor category. The following conditions are equivalent for an algebra $A$ in $\CC$:
\begin{enumerate}
\item[$(1)$] The algebra $A$ is exact.
\item[$(2)$] Every left $\CC$-module functor $F:\CC_A\to\CX$ is exact for every finite left $\CC$-module $\CX$.
\item[$(3)$] The finite monoidal category ${}_A\CC_A$ is a multi-tensor category.
\item[$(4)$] Every right exact left $\CC$-module functor $F:\CC_A\to\CC_A$ is exact.
\item[$(5)$] The functor $-\otimes_A y:\CC_A\to\CC$ is exact for every $y\in{}_A\CC$.
\item[$(6)$] The functor $x\otimes_A-:{}_A\CC\to\CC$ is exact for every $x\in\CC_A$.
\item[$(7)$] For any injective $I\in\CC$ and any object $x\in\CC_A$, the object $I\otimes x$ is injective.
\item[$(8)$] The functor $-\otimes_A-: \CX_A\times{}_A\CY\to\CX\boxtimes_\CC\CY$ is exact in each variable for every finite right $\CC$-module $\CX$ and finite left $\CC$-module $\CY$.
\end{enumerate}
\end{prop}

\begin{proof}
$(1)\Rightarrow(2)$ Let $C = (0\to x\xrightarrow{f}y\xrightarrow{g}z\to 0)$ be an exact sequence in $\CC_A$. Then for any projective $P\in\CC$, the exact sequence $P\otimes C$ consists of projective objects hence splits. Thus the sequence $P\odot F(C) \simeq F(P\otimes C)$ is exact. Therefore, $F(C)$ itself is exact.

$(2)\Rightarrow(3)$ Every left $\CC$-module functor $F:\CC_A\to\CC_A$ is exact hence has both a left adjoint and a right adjoint. Therefore, ${}_A\CC_A\simeq\Fun_\CC(\CC_A,\CC_A)^\rev$ is rigid.

$(3)\Rightarrow(4)$ Since ${}_A\CC_A\simeq\Fun_\CC(\CC_A,\CC_A)^\rev$ is rigid, $F$ has both a left adjoint and a right adjoint hence is exact.

$(4)\Rightarrow(5)$ Let $f: x\to x'$ be a monomorphism. Since the functor $-\otimes_A y\otimes z:\CC_A\to\CC_A$ is exact by assumption, the object $\Ker(f\otimes_A y)\otimes z \simeq\Ker(f\otimes_A y\otimes z)$ vanishes for all $z\in\CC_A$. Thus $\Ker(f\otimes_A y)$ itself vanishes, as desired.

$(5)\Rightarrow(1)$ For any projective $P\in\CC$ and any object $x\in\CC_A$, since the functor $[x,-]_\CC\simeq(-\otimes_A x^R)^L$ is exact by assumption, the functor $\Hom_{\CC_A}(P\otimes x,-) \simeq \Hom_\CC(P,[x,-]_\CC)$ is also exact. Namely, $P\otimes x$ is projective.

$(3)\Leftrightarrow(6)$ We have identifications ${}_A(\CC^\rev)_A = ({}_A\CC_A)^\rev$ and $(\CC^\rev)_A = {}_A\CC$. Applying $(3)\Leftrightarrow(5)$ to $\CC^\rev$ we obtain $(3)\Leftrightarrow(6)$.

$(3)\Leftrightarrow(7)$  We have an equivalence $(\CC^\rev)_A\simeq(\CC_A)^\op$, $x\mapsto x^L$. Applying $(3)\Leftrightarrow(1)$ to $\CC^\rev$ we obtain $(3)\Leftrightarrow(7)$.

$(8)\Rightarrow(5)(6)$ is trivial. 

$(5)(6)\Rightarrow(8)$ Suppose that $\CX={}_X\CC$ and $\CY=\CC_Y$ where $X,Y\in\Alg(\CC)$. Since the forgetful functors $\CX,\CY\to\CC$ respect exact sequence, we may assume without loss of generality that $\CX=\CY=\CC$. Hence (8) is reduced to (5)(6).
\end{proof}

\begin{defn}
An exact algebra $A$ in a multi-tensor category $\CC$ is {\em simple} if $A$ is a simple $A$-$A$-bimodule or, equivalently, ${}_A\CC_A$ is a tensor category.
\end{defn}

\void{
\begin{prop}
If $A$ is exact, then the functor $\CX_A\times{}_A\CY\to\CX\boxtimes_\CC\CY$, $(x,y)\mapsto x\otimes_A y$ is exact in each variable.
\end{prop}

\begin{proof}
Suppose that $\CX={}_X\CC$ and $\CY=\CC_Y$. Since the forgetful functors $\CX,\CY\to\CC$ respect exact sequence, we may assume without loss of generality that $\CX=\CY=\CC$. Then the functor $-\otimes_A y:\CC_A\to\CC$ is exact by \cite[Proposition 3.11]{eo}. Replacing $\CC$ by $\CC^\rev$ we see that $x\otimes_A-$ is also exact.
\end{proof}
}

In the dual picture, suppose that $A$ is a coalgebra in a multi-tensor category $\CC$ (i.e. an algebra in $\CC^\op$). We use $x\otimes^A y$ to denote the equalizer of the parallel morphisms $x\boxtimes_\CC y \rightrightarrows x\boxtimes_\CC (A\odot y)$ for a right $A$-comodule $x$ in a finite right $\CC$-module $\CX$ (i.e. a right $A$-module in the right $\CC^\op$-module $\CX^\op$) and a left $A$-comodule $y$ in a finite left $\CC$-module $\CY$. If $A$ is exact (as an algebra in $\CC^\op$), then the functor $(x,y)\mapsto x\otimes^A y$ is exact in each variable.% and, therefore, the induced functor $\RMod_A(\CX^\op)^\op\boxtimes{}_A\CY^\op^\op\to\CX\boxtimes_\CC\CY$ is exact.

\subsection{Formula for horizontal fusion of internal homs} \label{sec:formula}

%\subsection{The statement of the formula} \label{subsec:formula}

Let $\CM$ be a multi-tensor category, $\CX$ be a finite right $\CM$-module, and $\CU=\Fun_{\CM^\rev}(\CX,\CX)$. Then $\CU$ is a monoidal right $\FZ(\CM)$-module by Example \ref{exam:Fun_monoidal_bimodule}. Suppose that $M\in\Alg(\CM)$, $x\in\CX_M$. Then $\CX_M$ is a finite $U$-${}_M\CM_M$-bimodule and $x$ is an $[x,x]_\CU$-$\one_{{}_M\CM_M}$-bimodule in the obvious way. Since $x$ is a right $\one_{{}_M\CM_M}$-module, it is also a $Z(M)$-module via the algebra isomorphism $Z(M)\simeq Z_{\FZ(\CM)}(\one_{{}_M\CM_M})$. It follows that $x$ is a left $[x,x]_\CU \odot Z(M)$-module, which supplies an algebra homomorphism $m: [x,x]_\CU \odot Z(M) \to [x,x]_\CU$. Then, $m$ defines a unital $Z(M)$-action on $[x,x]_\CU$. In other words, $m$ equips $[x,x]_\CU$ with the structure of a right $Z(M)$-module in $\Alg(\CU)$.

More generally, let $\CL$, $\CM$ be multi-tensor categories, $\CX$ be a finite $\CL$-$\CM$-bimodule, and $\CU=\Fun_{\CL|\CM}(\CX,\CX)$. Then $x$ is a $Z(L)\boxtimes Z(M)$-module and $[x,x]_\CU$ is a $Z(L)$-$Z(M)$-bimodule in $\Alg(\CU)$ for $L\in\Alg(\CL)$, $M\in\Alg(\CM)$ and $x\in{}_L\CX_M$.

For a braided multi-tensor category $\CB$, a monoidal right $\CB$-module $\CU$, a monoidal left $\CB$-module $\CV$ and for $B\in\CAlg(\CB)$, $U\in\Alg(\CU)_B$, $V\in{}_B\Alg(\CV)$, the relative tensor product $U\otimes_B V$, which is defined by the coequalizer of the following two parallel morphisms:
\be \label{eq:coequalizer-B}
\raisebox{1.7em}{\xymatrix@R=0em{
& (U\odot B) \boxtimes_\CB V \ar[rd] & \\
U\boxtimes_\CB B \boxtimes_\CB V \ar[ur]^\simeq \ar[rd]_\simeq & & U \boxtimes_\CB V\, , \\
& U \boxtimes_\CB (B\odot V) \ar[ur] &
}}
\ee
has a unique structure of an algebra in $\CU\boxtimes_\CB\CV$ such that the projection $U\boxtimes_\CB V \twoheadrightarrow U\otimes_B V$ is an algebra homomorphism. The proof of this fact is entirely the same as that of \cite[Lemma 4.5]{dkr}. We omit the details.

\medskip

The main purpose of this subsection is to prove the following fusion formula:

\begin{thm} \label{thm:main_formula}
Let $\CL,\CM,\CN$ be indecomposable multi-tensor categories and $L,M,N$ be simple exact algebras in $\CL,\CM,\CN$, respectively. Let ${}_\CL\CX_\CM,{}_\CM\CY_\CN$ be finite bimodules and let $\CU=\Fun_{\CL|\CM}(\CX,\CX)$, $\CB=\FZ(\CM)$, $\CV=\Fun_{\CM|\CN}(\CY,\CY)$. 

$(1)$ There is a natural isomorphism for $x,x'\in{}_L\CX_M$ and $y,y'\in{}_M\CY_N$:
\begin{equation} \label{eqn:main}
[x,x']_\CU\otimes_{Z(M)}[y,y']_\CV \simeq [x\otimes_M y,x'\otimes_M y']_{\CU\boxtimes_\CB\CV}.
\end{equation}

$(2)$ If $x=x'$ and $y=y'$, then \eqref{eqn:main} is an algebra isomorphism.
\end{thm}

\begin{prop} \label{prop:fun-hom}
Let $\CL$ be a multi-tensor category and $L$ be an exact algebra in $\CL$. Let $\CX$ be a finite left $\CL$-module and let $\CU=\Fun_\CL(\CX,\CX)$.
We have a natural isomorphism $[x,x']_\CU \simeq [-,x]_\CL^R\otimes^{L^R} x'$ for $x,x'\in{}_L\CX$.% and $z\in\CX$.
\end{prop}

\begin{proof}
Suppose that $\CX=\CL_X$ and identify ${}_L\CX$ with ${}_L\CL_X$, $\CU$ with $({}_X\CL_X)^\rev$. Then $[x,x']_\CU \simeq x^R\otimes^{L^R} x'$. Thus $[x,x']_\CU \simeq -\otimes_X (x^R\otimes^{L^R} x') \simeq (-\otimes_X x^R)\otimes^{L^R} x' \simeq [-,x]_\CL^R\otimes^{L^R} x'$, where the second isomorphism is due to the exactness of $-\otimes^{L^R} x'$.
\end{proof}

\begin{rem}
In the situation of Proposition \ref{prop:fun-hom}, suppose that $\CX$ is a finite $\CL$-$\CM$-bimodule, where $\CM$ is another multi-tensor category, so that $\CU=\Fun_\CL(\CX,\CX)$ is an $\CM$-$\CM$-bimodule. Then $m\odot[x,x']_\CU \odot m' \simeq [x\odot m^L,x'\odot m']_\CU$ for $m,m'\in\CM$. 
\end{rem}

\begin{rem}
In the situation of Proposition \ref{prop:fun-hom}, we have $[x',x'']_\CU\circ[x,x']_\CU \simeq [-,x]_\CL^R\otimes^{L^R} [x',x']_\CL^R\otimes^{L^R} x''$ and the composition $[x',x'']_\CU\circ[x,x']_\CU\to[x,x'']_\CU$ is induced by the canonical morphism $L\to[x',x']_\CL$.
\end{rem}

\begin{rem}[Uniqueness of unit] \label{rem:unique-unit}
If $(A,u,m)$ and $(A,u',m)$ are two algebras in a monoidal category then $u=u'$. In fact, $u=m\circ(u\otimes u')=u'$. As a consequence, to show that an isomorphism is an algebra isomorphism, it suffices to verify that it preserves multiplication. 
\end{rem}

\begin{prop} \label{prop:exactalg}
Let $\CL$ be an indecomposable multi-tensor category and $L$ be a simple exact algebra in $\CL$. The monoidal equivalence $F: \CL\boxtimes_{\FZ(\CL)}\CL^\rev \to \Fun(\CL,\CL)$ maps the algebra $L\otimes_{Z(L)}L$ to $[L,L]_{\Fun(\CL,\CL)}$.
%Let $\CL$ be an indecomposable multi-tensor category and $L$ be a simple exact algebra in $\CL$. The $\CL^\rev\boxtimes\CL$-$\CL^\rev\boxtimes\CL$-bimodule equivalence $F: \CL\boxtimes_{\FZ(\CL)}\CL \to \Fun(\CL,\CL)$ maps the $L\boxtimes L$-$L\boxtimes L$-bimodule $L\otimes_{Z(L)}L$ to $[L,L]_{\Fun(\CL,\CL)}$.
\end{prop}

\begin{proof}
Let $\CC$ denote the tensor category ${}_L\CL_L$. Identify $\FZ(\CL)$ and $Z(L)$ with $\FZ(\CC)$ and $Z(\one_\CC)$, respectively.
Note that $F(L\otimes_{Z(L)}L)$ is the coequalizer of the parallel morphisms $Z(L)\otimes L\otimes-\otimes L \rightrightarrows L\otimes-\otimes L$. Rewrite the diagram as $(Z(L)\otimes L)\otimes_L (L\otimes-\otimes L) \rightrightarrows L\otimes-\otimes L$ and let us compute the coequalizer in $\CC$. Invoking Proposition \ref{prop:coequ}(2), we see that the coequalizer is $\Hom_\CC(L\otimes-\otimes L,L)^\vee\otimes L$. It is isomorphic to $\Hom_\CL(-,L)^\vee\otimes L$ and, by Proposition \ref{prop:fun-hom}, to $[L,L]_{\Fun(\CL,\CL)}$, as desired.

%The unit morphisms of both algebras $F(L\otimes_{Z(L)}L)$ and $[L,L]_{\Fun(\CL,\CL)}(x)$ are given by the canonical morphism $\Id_\CL\to\Hom_\CL(-,L)^\vee\otimes L$. 
The multiplication of $[L,L]_{\Fun(\CL,\CL)}$ is given by the morphism $\Hom_\CL(-,L)^\vee\otimes\Hom_\CL(L,L)^\vee\otimes L\to \Hom_\CL(-,L)^\vee\otimes L$ induced by the canonical map $\bk\to\Hom_\CL(L,L)$. That of $F(L\otimes_{Z(L)}L)$ is given by the morphism $\Hom_\CC(L\otimes-\otimes L,L)^\vee\otimes\Hom_\CC(L\otimes L\otimes L,L)^\vee\otimes L\to \Hom_\CL(L\otimes-\otimes L,L)^\vee\otimes L$ induced by the multiplication $L\otimes L\otimes L \to L$ and the equivalence $\Hom_\CC(L,L)\simeq\bk$. Therefore, $F(L\otimes_{Z(L)}L)$ and $[L,L]_{\Fun(\CL,\CL)}$ are isomorphic as algebras.
\end{proof}

\begin{lem} \label{lem:inhom}
Let $\CM$ be a multi-tensor category, $M'$ be an algebra in $\CM$, $M$ be an exact algebra in $\CM$, $\CX$ be a finite right $\CM$-module and $\CY$ be a finite left $\CM$-module. There is a natural isomorphism
$$\Hom_{\CX\boxtimes_\CM\CY}(x'\otimes_{M'} y',x\otimes_M y) \simeq \Hom_{{}_{M'}\CM_{M'}}(M',[x',x]_{\CM^\rev}\otimes_M[y',y]_\CM)$$
for $x'\in\CX_{M'}$, $y'\in{}_{M'}\CY$ and $x\in\CX_M$, $y\in{}_M\CY$.
\end{lem}

\begin{proof}
Suppose that $\CX={}_X\CM$, $\CY=\CM_Y$ where $X,Y\in\Alg(\CM)$. Then 
LHS $\simeq \Hom_{{}_X\CM_Y}(x'\otimes_{M'} y',x\otimes_M y) 
%\simeq \Hom_\CM(M',(s'^L\otimes_X s)^R\otimes_M(y'\otimes_Y y^R)^L)
\simeq \Hom_{{}_{M'}\CM_{M'}}(M',x'^R\otimes^{X^R}x\otimes_M y\otimes^{Y^L}y'^L)
\simeq$ RHS.
\end{proof}

\begin{proof}[Proof of Theorem \ref{thm:main_formula}(1)]
Let $\CW=\Fun_{\CL|\CN}(\CX\boxtimes_\CM\CY,\CX\boxtimes_\CM\CY)$. The equivalence $\CU \boxtimes_\CB \CV \simeq \CW$ maps $[x,x']_\CU \otimes_{Z(M)} [y,y']_\CV$ to 
\begin{equation*}
\begin{split}
[x,x']_\CU \otimes_{Z(M)} [y,y']_\CV 
\simeq ([-,x]_{\CL\boxtimes\CM^\rev}^R \otimes^{L^R\boxtimes M^L} x') \otimes_{Z(M)} ([-,y]_{\CM\boxtimes\CN^\rev}^R \otimes^{M^R\boxtimes N^L} y')
\end{split}
\end{equation*}
where the isomorphism is due to Proposition \ref{prop:fun-hom}.
On the other hand side,
\begin{equation*}
\begin{split}
[x\otimes_M y,x'\otimes_M y']_\CW 
\simeq & [-\boxtimes_\CM-,x\otimes_M y]_{\CL\boxtimes\CN^\rev}^R \otimes^{L^R\boxtimes N^L} (x'\otimes_M y') \\
\simeq & \Hom_\CM(\one_\CM,[-,x]_{\CL\boxtimes\CM^\rev}\otimes_M[-,y]_{\CM\boxtimes\CN^\rev})^R \otimes^{L^R\boxtimes N^L} (x'\otimes_M y')
\end{split}
\end{equation*}
where the last isomorphism is due to Lemma \ref{lem:inhom}.
Comparing the right hand sides of the above two equations, one observes that the theorem can be reduced to the special case $L=\one_\CL$ and $N=\one_\CN$. 
Then notice that both expressions depend only on $[-,x]_{\CL\boxtimes\CM^\rev}$, $[-,y]_{\CM\boxtimes\CN^\rev}$, $x'\simeq x'\otimes_M M$, $y'\simeq M\otimes_M y'$ and $M$, therefore it suffices to show that
$$(u^R\otimes^{M^L}M)\otimes_{Z(M)}(v^R\otimes^{M^R}M) \simeq \Hom_\CM(\one_\CM,u\otimes_M v)^R\boxtimes M$$
in $\CL\boxtimes\CM\boxtimes\CN$ for $u\in(\CL\boxtimes\CM^\rev)_M$ and $v\in{}_M(\CM\boxtimes\CN^\rev)$. This reduces the theorem to the special case $\CL=\CN=\bk$, $\CX=\CY=\CM$ and $x=x'=y=y'=M$. 
This special case is exactly Proposition \ref{prop:exactalg}.
\end{proof}

\begin{proof}[Proof of Theorem \ref{thm:main_formula}(2)]
For simplicity we assume $L=\one_\CL$ and $N=\one_\CN$. We need to show that the isomorphism \eqref{eqn:main} is compatible with the multiplications of the two algebras (see Remark \ref{rem:unique-unit}). This amounts to show the commutativity of the outer square of the following digram for $u\in(\CL\boxtimes\CM^\rev)_M$ and $v\in{}_M(\CM\boxtimes\CN^\rev)$
$$\scriptsize\xymatrix@R=2em{
 (u^R\otimes^{M^L}[x,x]^R\otimes^{M^L}M) \otimes_{Z(M)\otimes Z(M)} (v^R\otimes^{M^R}[y,y]^R\otimes^{M^R}M) \ar[r]^-\sim \ar[r]^-\sim \ar[d]_\alpha & \Hom_\CM(\one_\CM,u\otimes_M v)^R \otimes \Hom_{{}_M\CM_M}(M,[x,x]\otimes_M[y,y])^R\boxtimes M \ar[d]^\beta \\
 (u^R\otimes^{M^L}M^L\otimes^{M^L}M) \otimes_{Z(M)\otimes Z(M)} (v^R\otimes^{M^R}M^R\otimes^{M^R}M) \ar[r]^-\sim \ar[r]^-\sim \ar[d]_\gamma & \Hom_\CM(\one_\CM,u\otimes_M v)^R \otimes \Hom_{{}_M\CM_M}(M,M)^R\boxtimes M \ar[d]^\delta \\
 (u^R\otimes^{M^L}M) \otimes_{Z(M)} (v^R\otimes^{M^R}M) \ar[r]^-\sim & \Hom_\CM(\one_\CM,u\otimes_M v)^R\boxtimes M \\ 
}$$
where $\alpha,\beta$ are induced by the canonical morphisms $L\boxtimes M\to[x,x]_{\CL\boxtimes\CM^\rev}$ and $M\boxtimes N\to[y,y]_{\CM\boxtimes\CN^\rev}$, $\gamma$ induced by the multiplication morphism $Z(M)\otimes Z(M)\to Z(M)$, $\delta$ induced by the equivalence $\Hom_{{}_M\CM_M}(M,M)\simeq\bk$.
It is clear that the upper square of the diagram is commutative. Proposition \ref{prop:exactalg} states that the theorem holds for the special case $\CL=\CN=\bk$, $\CX=\CY=\CM$ and $x=x'=y=y'=M$, therefore the lower square is commutative. This completes the proof.
\end{proof}

\subsection{Pointed Drinfeld center functor} \label{sec:Df-center-functor}

We introduce two symmetric monoidal categories $\mtc_\bullet$ and $\btc_\bullet$ as follows:
\begin{itemize}
\item 
An object of $\mtc_\bullet$ is a pair $(\CL,L)$ where $\CL$ is an indecomposable multi-tensor category over $k$ and $L$ is a simple exact algebra in $\CL$.
A morphism between two objects $(\CL,L)$ and $(\CM,M)$ is an equivalence class of pairs $(\CX,x)$ where $\CX$ is a finite $\CL$-$\CM$-bimodule and $x\in{}_L\CX_M$; two pairs $(\CX,x)$ and $(\CX',x')$ are equivalent if there exist a bimodule equivalence $F:\CX\to\CX'$ and a bimodule isomorphism $F(x)\simeq x'$.
The composition of two morphisms $(\CX,x):(\CL,L)\to(\CM,M)$ and $(\CY,y):(\CM,M)\to(\CN,N)$ is given by $(\CX\boxtimes_\CM\CY,x\otimes_M y)$.
\item
An object of $\btc_\bullet$ is a pair $(\CA,A)$ where $\CA$ is a braided tensor category over $k$ and $A$ is a commutative algebra in $\CA$.
A morphism between two objects $(\CA,A)$ and $(\CB,B)$ is an equivalence class of pairs $(\CU,U)$ where $\CU$ is a monoidal $\CA$-$\CB$-bimodule and $U\in{}_A\Alg(\CU)_B$; two pairs $(\CU,U)$ and $(\CU',U')$ are equivalent if there exist a monoidal bimodule equivalence $F:\CU\to\CU'$ and a bimodule isomorphism $F(U)\simeq U'$. 
The composition of two morphisms $(\CU,U):(\CA,A)\to(\CB,B)$ and $(\CV,V):(\CB,B)\to(\CC,C)$ is given by $(\CU\boxtimes_\CB\CV,U\otimes_B V)$.
\end{itemize}
The tensor product functors of both categories are Deligne's tensor product $\boxtimes$.

\medskip

The following theorem generalizes Theorem \ref{thm:Z-functor}:

\begin{thm} \label{thm:Df-center}
The assignment 
$$(\CL,L)\mapsto(\FZ(\CL),Z(L)), \quad ({}_\CL\CX_\CM,{}_L x_M) \mapsto (\FZ^{(1)}(\CX):=\Fun_{\CL|\CM}(\CX,\CX),Z^{(1)}(x):=[x,x]_{\FZ^{(1)}(\CX)})$$ 
defines a symmetric monoidal functor 
$$\FZ:\mtc_\bullet\to\btc_\bullet.$$
\end{thm}

\begin{proof}
By Corollary \ref{cor:center-A}(1), $\FZ$ preserves identity morphism. By Theorem \ref{thm:Z-functor} and Theorem \ref{thm:main_formula}, $\FZ$ preserves composition law.
\end{proof}

Note that $(\FZ(\CL),Z(L))$ is indeed the center of $(\CL,L)$ as shown in \cite{s2}. We will refer to this functor $\FZ$ as the {\em pointed Drinfeld center functor}.

\begin{prop} \label{prop:cll}
Let $(\CL,L)$ be an object of $\mtc_\bullet$. 
$(1)$ The morphism $(\CL_L,L):(\CL,L)\to({}_L\CL_L,L)$ is inverse to $({}_L\CL,L):({}_L\CL_L,L)\to(\CL,L)$.
$(2)$ The pairs $\FZ(\CL,L)$ and $\FZ({}_L\CL_L,L)$ are canonically identified so that $\FZ(\CL_L,L)$ is the identity morphism.
\end{prop}

\begin{proof}
(1) It is clear that ${}_L\CL\boxtimes_\CL\CL_L \simeq {}_L\CL_L$. According to \cite[Theorem 3.27]{eo}, ${}_L\CL\boxtimes_{{}_L\CL_L}\CL_L \simeq \CL$. Namely, the $\CL$-${}_L\CL_L$-bimodule $\CL_L$ is inverse to ${}_L\CL$. Then the claim follows from the trivial isomorphism $L\otimes_L L\simeq L$.

(2) According to \cite[Theorem 3.34, Corollary 3.35 and Remark 3.36]{eo}, there are canonical monoidal equivalences $\FZ(\CL) \simeq \FZ^{(1)}(\CL_L) \simeq \FZ({}_L\CL_L)$ which induce a braided monoidal equivalence $\FZ(\CL) \simeq \FZ({}_L\CL_L)$. Moreover, $Z(L)\simeq [L,L]_{\FZ(\CL)} \simeq Z(\one_{{}_L\CL_L})$ by Corollary \ref{cor:center-A}.
\end{proof}

In the rest of this subsection, we assume $\chara k=0$.

\begin{defn}[\cite{dmno}]
A {\em Lagrangian algebra} in a non-degenerate braided fusion category $\CC$ is a commutative simple separable algebra $A$ such that the category of local $A$-modules in $\CC$ is equivalent to $\bk$.
\end{defn}

We introduce two symmetric monoidal subcategories $\mfus_\bullet\subset\mtc_\bullet$ and $\bfuscl_\bullet\subset\btc_\bullet$ as follows:
\begin{itemize}
\item 
An object of $\mfus_\bullet$ is a pair $(\CL,L)$ where $\CL$ is an indecomposable multi-fusion category and $L$ is a simple separable algebra in $\CL$.
A morphism between two objects $(\CL,L)$ and $(\CM,M)$ is an equivalence class of pairs $(\CX,x)$ where $\CX$ is a nonzero semisimple $\CL$-$\CM$-bimodule and $x$ is a nonzero $L$-$M$-bimodule in $\CX$. 
\item
An object of $\bfuscl_\bullet$ is a pair $(\CA,A)$ where $\CA$ is a non-degenerate braided fusion category and $A$ is a Lagrangian algebra in $\CA$.
A morphism between two objects $(\CA,A)$ and $(\CB,B)$ is an equivalence class of pairs $(\CU,U)$ where $\CU$ is a closed multi-fusion $\CA$-$\CB$-bimodule and $U$ is a closed $A$-$B$-bimodule in the category of separable algebras in $\CU$ (see Definition \ref{def:closed-AB-bimodule}). 
\end{itemize}
It is clear that $\mfus_\bullet$ is a well-defined subcategory, i.e. the class of morphisms in $\mfus_\bullet$ contains identity morphisms and is closed under composition. However, this is not obvious for $\bfuscl_\bullet$.

\begin{rem} \label{rem:ff}
Any object $(\CA,A)$ of $\bfuscl_\bullet$ has the form $\FZ(\CL,\one_\CL)$ where $\CL$ is a fusion category. In particular, $\CA$ is non-chiral. Indeed, we have $\CA\simeq\FZ(\CC_A)$ by \cite[Corollary 4.1(i)]{dmno} because $A$ is a Lagrangian algebra. Therefore, $(\CA,A)$ can be identified with $\FZ(\CA_A,A)$. 
%For a morphism $(\CU,U)$ of $\bfuscl_\bullet$, the separable algebra $U$ is simple because its center is simple.
\end{rem}

\begin{prop} 
The functor $\FZ:\mtc_\bullet\to\btc_\bullet$ maps objects in $\mfus_\bullet$ into $\bfuscl_\bullet$.\end{prop}

\begin{proof}
Let $(\CL,L)$ be an object of $\mfus_\bullet$. We need to show that $Z(L)$ is a Lagrangian algebra.
In view of Proposition \ref{prop:cll}, replacing $(\CL,L)$ by $({}_L\CL_L,L)$ if necessary we may assume that $\CL$ is a fusion category and $L=\one_\CL$. We have $Z(\one_\CC) \simeq [\one_\CC,\one_\CC]_{\FZ(\CC)}$ by Corollary \ref{cor:center-A}(2), while the proof of \cite[Proposition 4.1]{dmno} showed that $[\one_\CC,\one_\CC]_{\FZ(\CC)}$ is a Lagrangian algebra.  
\end{proof}

\begin{prop} \label{prop:ff1}
The functor $\FZ:\mtc_\bullet\to\btc_\bullet$ maps morphisms in $\mfus_\bullet$ into $\bfuscl_\bullet$.
\end{prop}

\begin{proof}
Let $(\CX,x):(\CL,L)\to(\CM,M)$ be a morphism in $\mfus_\bullet$. Since ${}_L\CX_M$ is a semisimple left module over $\CU:=\Fun_{\CL|\CM}(\CX,\CX)$, the algebra $[x,x]_\CU$ is separable. We need to show that $Z(L)\boxtimes Z(M)$ is the full center of $[x,x]_\CU$.

First, replacing $(\CL,L)$ by $(\bk,k)$ and $(\CM,M)$ by $(\CL^\rev\boxtimes\CM,L\boxtimes M)$ if necessary, we may assume $(\CL,L)=(\bk,k)$.
Then in view of Proposition \ref{prop:cll}, replacing $(\CM,M)$ by $\FZ({}_M\CM_M,M)$ and $(\CX,x)$ by $(\CX_M,x)$ if necessary, we may assume that $\CM$ is a fusion category and $M=\one_\CM$.
Finally, since $\CX$ is an indecomposable left $\CU$-module, we have $\CU_{[x,x]}\simeq\CX$, thus ${}_{[x,x]}\CU_{[x,x]} \simeq \CM$ by \cite[Theorem 3.27]{eo}. Therefore, $Z([x,x]_\CU) \simeq Z(\one_\CM)$ by Corollary \ref{cor:center-A}(3), as desired.
\end{proof}

\begin{lem} \label{lem:ff2}
The following induced map is bijective for fusion categories $\CL,\CM$
$$\phi:\Hom_{\mfus_\bullet}((\CL,\one_\CL),(\CM,\one_\CM)) \to \Hom_{\bfuscl_\bullet}(\FZ(\CL,\one_\CL),\FZ(\CM,\one_\CM)).$$
\end{lem}

\begin{proof}
By the folding trick, we may assume $\CL=\bk$. We need to construct an inverse to the map $\phi$.
Let $(\CU,U):\FZ(\bk,k)\to\FZ(\CM,\one_\CM)$ be a morphism in $\bfuscl_\bullet$ so that the pair $(\FZ(\CU),Z(U))$ can be identified with $(\FZ(\CM),Z(\one_\CM))$. Since $\FZ(\CU)_{Z(U)} \simeq {}_U\CU_U$ and $\FZ(\CM)_{Z(\one_\CM)} \simeq \CM$, ${}_U\CU_U$ can be identified with $\CM$, hence we obtain a morphism $(\CU_U,U):(\bk,k)\to(\CM,\one_\CM)$ in $\mfus_\bullet$.
It is clear that $\FZ^{(1)}(\CU_U,U)\simeq(\CU,U)$. Conversely, given a morphism $(\CX,x):(\bk,k)\to(\CM,\one_\CM)$ in $\mfus_\bullet$, we set $(\CU,U)=\FZ^{(1)}(\CX,x)$, i.e. $(\CU,U)=(\Fun_{\bk|\CM}(\CX,\CX),[x,x]_\CU)$. Since $\CX$ is an indecomposable left $\CU$-module, we have an equivalence $\CU_U\simeq\CX$ which maps $U$ to $x$. Thus $(\CU_U,U)$ represents the same morphism as $(\CX,x)$.
\end{proof}

\begin{cor} \label{cor:ff}
The subcategory $\bfuscl_\bullet$ is well-defined.
\end{cor}

\begin{proof}
According to Remark \ref{rem:ff} and Lemma \ref{lem:ff2}, the class of morphisms in $\bfuscl_\bullet$ contains identity morphisms and is closed under composition.
\end{proof}

The following theorem generalize Corollary \ref{cor:KZ-3}.

\begin{thm} \label{thm:ff}
The pointed Drinfeld center functor $\FZ:\mtc_\bullet\to\btc_\bullet$ restricts to a symmetric monoidal equivalence $\mfus_\bullet \simeq \bfuscl_\bullet$.
\end{thm}

\begin{proof}
According to Remark \ref{rem:ff}, the restricted functor $\mfus_\bullet \to \bfuscl_\bullet$ is essentially surjective. It remains to show that $\FZ$ is fully faithful, i.e. $\FZ$ induces a bijection 
$$\Hom_{\mfus_\bullet}((\CL,L),(\CM,M)) \simeq \Hom_{\bfuscl_\bullet}(\FZ(\CL,L),\FZ(\CM,M))$$
for $(\CL,L),(\CM,M)\in\mfus_\bullet$.
Since $(\CL,L)\simeq({}_L\CL_L,L)$ in $\mfus_\bullet$ via the invertible morphism $(\CL_L,L)$ by Proposition \ref{prop:cll}, we may assume that $\CL$ is a fusion category and $L=\one_\CL$, and similarly for $(\CM,M)$. Then apply Lemma \ref{lem:ff2}.
\end{proof}

\subsection{Corollaries}

Assume $\chara k=0$. We derive some corollaries of Theorem \ref{thm:ff} in this subsection. The following corollary generalize the main result in \cite[Eq.\,(3.13)]{dkr1}.
\begin{cor} \label{cor:pic=aut}
Let $A$ be a simple separable algebra in a multi-fusion category $\CC$. We have a group isomorphism 
\begin{equation} \label{eqn:pic-aut}
\Pic(A) \simeq \Aut(Z(A))
\end{equation}
where $\Pic(A)$ is the group of isomorphism classes of invertible $A$-$A$-bimodules in $\CC$ and $\Aut(Z(A))$ is the automorphism group of the commutative algebra $Z(A)$.
\end{cor}

\begin{proof}
We may assume $\CC$ is indecomposable.
According to Proposition \ref{prop:cll}, we have $\Pic(A)\simeq\Pic(\one_{{}_A\CC_A})$ and $\Aut(Z(A))\simeq\Aut(Z(\one_{{}_A\CC_A}))$. Therefore, we may assume that $\CL$ is a fusion category and $A=\one_\CL$.
Let $x$ be an invertible $\one_\CC$-$\one_\CC$-bimodule in $\CC$, i.e. an invertible object of $\CC$. Then the left $Z(\one_\CC)$-action on $[x,x]_{\FZ(\CC)}$ is induced by $[x,x]_{\FZ(\CC)} \simeq [\one_\CC,x^R\otimes x]_{\FZ(\CC)} \simeq [\one_\CC,\one_\CC]_{\FZ(\CC)}$ and the right $Z(\one_\CC)$-action is induced by $[x,x]_{\FZ(\CC)} \simeq [\one_\CC,x\otimes x^L]_{\FZ(\CC)} \simeq [\one_\CC,\one_\CC]_{\FZ(\CC)}$. Hence the $Z(\one_\CC)$-$Z(\one_\CC)$-bimodule structure on $[x,x]_{\FZ(\CC)}$ induces an automorphism of $Z(\one_\CC)$.
Invoking Theorem \ref{thm:ff}, we see the map $\Pic(\one_\CC) \to \Aut(Z(\one_\CC))$ constructed above is a group isomorphism.
\end{proof}

\begin{rem}
When $\CC$ is a modular tensor category, above corollary was proved in \cite{dkr1}, and non-trivial examples of $\Pic(A)$ were also provided there. Another related earlier result is \cite[Theorem O]{frs2}.
\end{rem}

The following theorem reformulates and generalizes \cite[Theorem\,1.1]{kr1} and \cite[Proposition 4.8]{dmno}. 

\begin{cor} 
Two separable algebras $A$ and $B$ in a multi-fusion category $\CC$ are Morita equivalent if and only if they share the same full center, i.e. $Z(A) \simeq Z(B)$ as algebras in $\FZ(\CC)$. 
\end{cor}

\begin{proof}
One easy direction is immediate from Corollary \ref{cor:morita-invariant}. To see the other direction we suppose that $Z(A)\simeq Z(B)$. According to Corollary \ref{cor:center-sum}, we may assume without loss of generality that $\CC$ is indecomposable and that $A,B$ are simple. Then Theorem \ref{thm:ff} implies that $(\CC,A) \simeq (\CC,B)$ in $\mfus_\bullet$. That is, $A$ and $B$ are Morita equivalent.
\end{proof}

Restricting to the subcategory of $\mfus_\bullet$ consisting of objects in the form $(\CL,L)$ and morphisms in the form $(\CL,x)$, where $\CL$ is a fixed fusion category, 
%and morphisms in the form $(\CX,x)$, where $\CX$ are invertible $\CL$-$\CL$-bimodules, 
the pointed Drinfeld center functor recovers the 1-truncation of the full center 2-functor defined in \cite[Theorem 7.10]{dkr}. Moreover, our results strengthen it as follows. 

\begin{cor}
The 1-truncation of the full center 2-functor defined in \cite[Theorem 7.10]{dkr} is faithful. 
\end{cor}

\section{3-functors and physical meanings} \label{sec:cft}
%\section{Motivations from physics} \label{sec:cft}

Assume $\chara k=0$. We sketch a construction that promotes the Drinfeld center 1-functor $\FZ:\mfus \to \bfuscl$ tautologically to a 3-functor and the pointed Drinfeld center 1-functor $\FZ:\mfus_\bullet \to \bfuscl_\bullet$ tautologically to a 3-equivalence. In this section, we use $n$d to represent a spatial dimension and $n$D to represent a spacetime dimension.

\subsection{Drinfeld center as a 3-functor} \label{sec:d-3fun}

First, we promote $\mfus$ and $\bfuscl$ to symmetric monoidal (weak) 3-categories $\widehat\mfus$ and $\widehat\bfuscl$ as follows.
\begin{enumerate}
\item The 3-category $\widehat\mfus$:
\begin{itemize}
\item An object is a an indecomposable multi-fusion category $\CL$.
\item A 1-morphism between two objects $\CL$ and $\CM$ is a nonzero semisimple $\CL$-$\CM$-bimodule  $\CX$. 
%The composition of two morphisms $(\CX,x):(\CL,L)\to(\CM,M)$ and $(\CY,y):(\CM,M)\to(\CN,N)$ is given by $(\CX\boxtimes_\CM\CY,x\otimes_M y)$.
\item A 2-morphism between two 1-morphisms $\CX,\CX'$ is a bimodule functor $F:\CX\to\CX'$.
\item A 3-morphism between two 2-morphisms $F,F'$ is a bimodule natural transformation $\phi:F\to F'$.
\end{itemize}

\item The 3-category $\widehat\bfuscl$:
\begin{itemize}

\item An object is a non-degenerate braided fusion category $\CA$.

\item A 1-morphism between two objects $\CA,\CB$ is a closed multi-fusion $\CA$-$\CB$-bimodule $\CU$.  

\item A 2-morphism between two 1-morphisms $\CU,\CU': \CA\to\CB$ is a pair $(\CP,p)$ where 
$\CP$ is a semisimple left $\CU'\boxtimes_{\overline{\CA}\boxtimes\CB}\CU^\rev$-module (in particular, $\CP$ is a $\CU'$-$\CU$-bimodule) which is closed in the sense that the associated monoidal functor $\CU'\boxtimes_{\overline{\CA}\boxtimes\CB}\CU^\rev \to \Fun(\CP,\CP)$ is an equivalence,\footnote{One can show that $\CU'\boxtimes_{\overline{\CA}\boxtimes\CB}\CU^\rev$ is a matrix multi-fusion category, i.e. a multi-fusion category in the form $\Fun(\bk^n,\bk^n)$. Therefore, $\CP$ exists and is unique up to equivalence. Moreover, the category of left module functors of $\CP$ is equivalent to $\bk$.}
and $p\in\CP$ is a distinguished object.
The composition of $(\CP,p):\CU\to\CU'$ and $(\CQ,q):\CU'\to\CU''$ is given by $(\CQ\boxtimes_{\CU'}\CP,q\boxtimes_{\CU'}p)$.

\item A 3-morphism between two 2-morphisms $(\CP,p),(\CP',p'):\CU\to\CU'$ is an isomorphism class of pairs $(H,\phi)$ where
$H:\CP\to\CP'$ is a left module equivalence and $\phi:H(p)\to p'$ is a morphism in $\CP'$.
Two pairs $(H,\phi)$ and $(H',\phi')$ are isomorphic if there is a left module natural isomorphism $\eta:H\to H'$ such that $\phi=\phi'\circ\eta_p$.\footnote{By the previous footnote, $H$ is unique up to isomorphisms and $\eta$ is unique up to scalars.}

\end{itemize}
\end{enumerate}

Note that for a 2-morphism $(\CP,p)$ in $\widehat\bfuscl$ the first item $\CP$ is essentially redundant because it is unique up equivalence. Similarly, for a 3-morphism $(H,\phi)$ the first item $H$ is essentially redundant because it is unique up to isomorphism. Moreover, two pairs $(\id_\CP,\phi)$ and $(\id_\CP,\phi')$ represent the same 3-morphism $(\CP,p)\to(\CP,p')$ if and only if $\phi$ and $\phi'$ differ by an invertible scalar.
Therefore, the following assignment defines a 3-functor 
$$\widehat\FZ: \widehat\mfus \to \widehat\bfuscl$$
$$%\be \label{eq:3-functor}
\small 
\xymatrix{ \CL \ar@/^12pt/[rr]^{\CX} \ar@/_12pt/[rr]_{\CX'} & \Downarrow F \overset{\phi}{\Rrightarrow} F' \Downarrow 
& \CM}
\quad \longmapsto \quad
\xymatrix{ \FZ(\CL) \ar@/^16pt/[rr]^{\FZ^{(1)}(\CX):=\Fun_{\CL|\CM}(\CX,\CX)} \ar@/_16pt/[rr]_{\FZ^{(1)}(\CX'):=\Fun_{\CL|\CM}(\CX',\CX')} & \Downarrow \FZ^{(2)}(F) \overset{\FZ^{(3)}(\phi)}{\Rrightarrow} \FZ^{(2)}(F') \Downarrow 
& \FZ(\CM)}
$$%\ee
where $\FZ^{(2)}(F) = (\Fun_{\CL|\CM}(\CX,\CX'),F)$ and $\FZ^{(3)}(\phi) = (\id,\phi)$.

\begin{figure}[tb] 
\centering
    \includegraphics[scale=0.8]{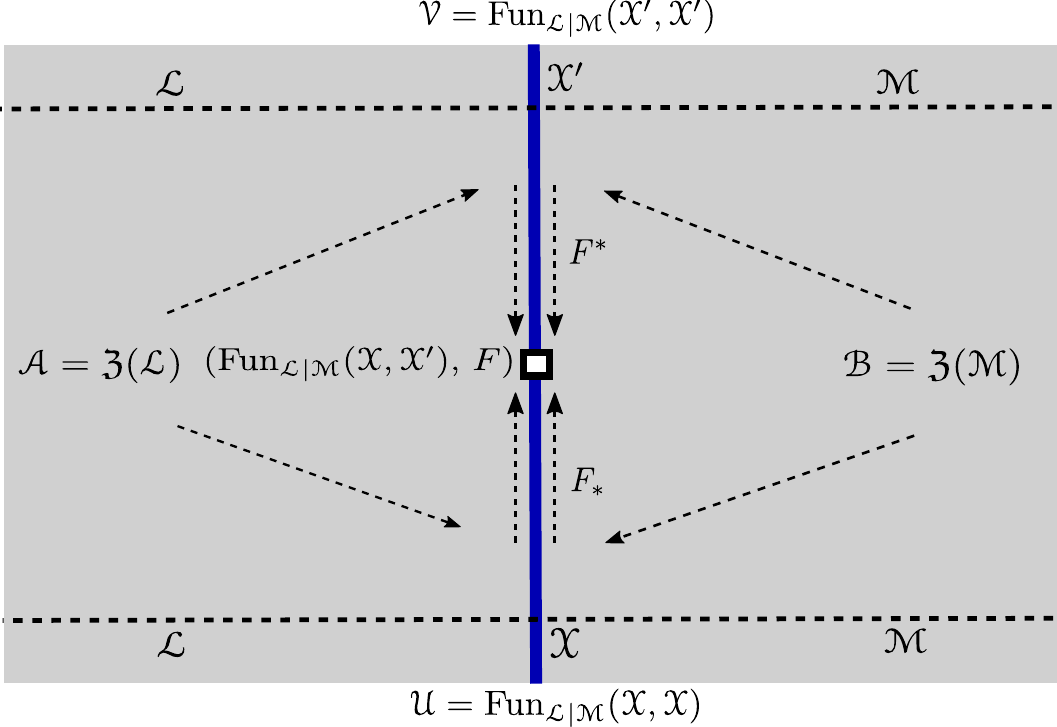}
    \caption{This figure illustrates the physical meaning of the image of the 2-truncated fully faithful functor $\widehat\FZ$, where $F\in \Fun_{\CL|\CM}(\CX,\CX')$.}
    \label{fig:drinfeld-center}
\end{figure}

The physical meaning of the image of $\widehat\FZ$ is illustrated by Figure\,\ref{fig:drinfeld-center} without drawing 3-morphisms. More precisely, in physical applications, $\CL$ and $\CM$ are unitary fusion categories; $\FZ(\CL)$ and $\FZ(\CM)$ are two unitary modular tensor categories describing two non-chiral 2d (the spatial dimension) topological orders; $\CU$ and $\CV$ are two 1d gapped domain walls between $\FZ(\CL)$ and $\FZ(\CM)$; the pair $(\Fun_{\CL|\CM}(\CX,\CX'), F)$, where $F\in \Fun_{\CL|\CM}(\CX,\CX')$, defines a 0d wall between $\CU$ and $\CV$; the dotted lines are functors, which describe how particle-like topological excitations in the 2d bulks (resp. 1d walls) are mapped into those on the 1d walls (resp. 0d wall); 3-morphisms $(\id,\phi)$ are instantons living on the time axis. 

%Note that the induced 3-functor $\widehat\mfus/k^\times \to \widehat\bfuscl$ is fully faithful where $\widehat\mfus/k^\times$ is obtained from $\widehat\mfus$ by identifying 3-morphisms up to scalars. In particular, the 2-truncation of $\FZ: \widehat\mfus \to \widehat\bfuscl$ is fully faithful. 

We can truncate both the domain and the codomain of $\widehat\FZ$ to 2-categories \cite{bicategory} by taking the isomorphism classes of 2-morphisms in the 3-category as 2-morphisms in the truncated 2-category. Such obtained 2-truncation of $\widehat\FZ$ is fully faithful. Similarly, we have the 1-truncation of $\widehat\FZ$, which is precisely the Drinfeld center 1-functor $\FZ:\mfus \to \bfuscl$. Its physical meaning is the complete boundary-bulk relation of 2d (i.e. a spatial dimension) topological order with gapped boundaries as illustrated in Figure\,\ref{fig:bbr-gapped} \cite{kwz}. More precisely, 
$\CL,\CM,\CN$ are unitary fusion categories describing three 1d boundaries; $\CX$ is a non-zero finite unitary $\CL$-$\CM$-bimodule and $\CP$ is a non-zero finite unitary $\CM$-$\CN$-bimodule describing two 0d defect junctions; $\FZ(\CL), \FZ(\CM), \FZ(\CN)$ are unitary modular tensor categories describing three 2d non-chiral topological orders; $\FZ^{(1)}(\CX)=\Fun_{\CL|\CM}(\CX,\CX)$ and $\FZ^{(1)}(\CP)=\Fun_{\CM|\CN}(\CP,\CP)$ are unitary multi-fusion categories describing two potentially unstable 1d gapped domain walls \cite{kwz}. The functoriality of the Drinfeld center says that the horizontal fusion of $\CX$ and $\CP$ on the boundary is compatible with that of $\FZ^{(1)}(\CX)$ and $\FZ^{(1)}(\CP)$ in the bulk.

\begin{figure}[tb]%[htpb]
 \begin{picture}(150, 100)
   \put(120,10){\scalebox{2}{\includegraphics{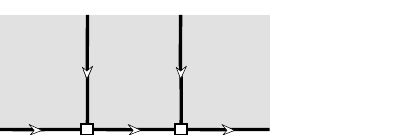}}}
   \put(80,-55){
     \setlength{\unitlength}{.75pt}\put(-18,-19){
     \put(95, 98)       {\scriptsize $\CL$}
     \put(170, 98)       {\scriptsize $\CM$}
     \put(245, 98)       {\scriptsize $\CN$}
     %\put(325, 98)       {\scriptsize $\CO$}
     \put(135,97)      {\scriptsize $\CX$}
     \put(208,97)      {\scriptsize $\CP$}
     %\put(285,95)      {\scriptsize $(\CZ,z)$}
     \put(85, 160)    {\scriptsize $\FZ(\CL)$}
     \put(165, 160)    {\scriptsize $\FZ(\CM)$}
     \put(245, 160)    {\scriptsize $\FZ(\CN)$}
     %\put(320, 160)    {\scriptsize $\FZ(\CO)$}
     \put(130,205)     {\scriptsize $\FZ^{(1)}(\CX)$}
     \put(202,205)     {\scriptsize $\FZ^{(1)}(\CP)$}
     %\put(285,212)     {\scriptsize $\FZ^{(1)}(\CZ)$}
     }\setlength{\unitlength}{1pt}}
  \end{picture}
\caption{The picture depicts the boundary-bulk relation of 2+1D topological orders. The arrows indicate the orientation of the boundaries or walls and the order of tensor product of topological excitations on the boundaries or walls. 
}
\label{fig:bbr-gapped}
\end{figure}
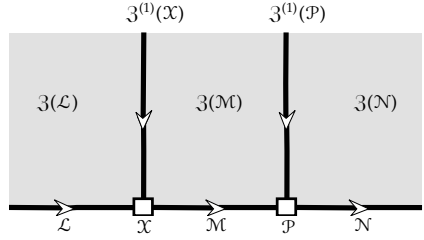

\void{
A fusion category $\CL$, as an object of $\widehat\mfus$, describes a potentially anomalous 1+1D gapped topological order. A nonzero semisimple $\CL$-$\CM$-bimodule $\CX$, i.e. a morphism in $\widehat\mfus$, describes a potentially anomalous gapped 0+1D wall. A bimodule functor $F:\CX\to\CX'$, i.e. a 2-morphism in $\widehat\mfus$, does not have a clear physical meaning, but its image does. %describes a transition of gapped 0+1D wall.

A braided fusion category $\CA$, i,e. an object of $\widehat\bfuscl$, describes an anomaly-free 2+1D gapped topological order. A closed $\CA$-$\CB$-bimodule, i.e. a morphism in $\widehat\bfuscl$ describes an anomaly-free 1+1D gapped wall. A 2-morphism $(\CP,F)$ describes a 0+1D gapped defect. A 3-morphism $(\id_\CP,\phi):(\CP,F)\to(\CP,F')$ describes an instanton, i.e. a 0D defect.

The 3-functor $\widehat\FZ: \widehat\mfus \to \widehat\bfuscl$ maps a 1+1D gapped topological order describe by a fusion category $\CC$ to its bulk described by $\FZ(\CC)$, maps a gapped 0+1D wall described by an $\CL$-$\CM$-bimodule $\CX$ to a 1+1D gapped wall described by $\FZ^{(1)}(\CX)$. A transition of gapped 0+1D wall on the boundary described by a bimodule functor $F:\CX\to\CX'$ is obtained by fusing the 0+1D defect $\FZ^{(2)}(F)$ in the bulk onto the boundary. The instanton $\FZ^{(3)}(\phi):\FZ^{(2)}(F)\to\FZ^{(2)}(F')$ induces a bimodule natural transformation $\phi:F\to F'$
}

\subsection{Pointed Drinfeld center as a 3-equivalence} \label{sec:df-3equiv}

By elaborating the construction of the previous subsection, we promote $\mfus_\bullet$ and $\bfuscl_\bullet$ to symmetric monoidal 3-categories $\widehat\mfus_\bullet$ and $\widehat\bfuscl_\bullet$ as follows.
\begin{enumerate}
\item The 3-category $\widehat\mfus_\bullet$:
\begin{itemize}
\item An object is a pair $(\CL,L)$, where $\CL$ is an indecomposable multi-fusion category and $L$ is a simple separable algebra in $\CL$.
\item A morphism between two objects $(\CL,L)$ and $(\CM,M)$ is a pair $(\CX,x)$, where $\CX$ is a nonzero semisimple $\CL$-$\CM$-bimodule and $x$ is a nonzero $L$-$M$-bimodule in $\CX$. 
%The composition of two morphisms $(\CX,x):(\CL,L)\to(\CM,M)$ and $(\CY,y):(\CM,M)\to(\CN,N)$ is given by $(\CX\boxtimes_\CM\CY,x\otimes_M y)$.
\item A 2-morphism between two morphisms $(\CX,x),(\CX',x') : (\CL,L)\to(\CM,M)$ is a pair $(F,f)$, where $F:\CX\to\CX'$ is an $\CL$-$\CM$-bimodule functor and $f:F(x)\to x'$ is an $L$-$M$-bimodule map.
\item A 3-morphism between two 2-morphisms $(F,f),(F',f'):(\CX,x)\to(\CX',x')$ is a bimodule natural transformation $\phi:F\to F'$ such that $f=f'\circ\phi_x$.
\end{itemize}

\item The 3-category $\widehat\bfuscl_\bullet$:
\begin{itemize}
\item An object is a pair $(\CA,A)$, where $\CA$ is a non-degenerate braided fusion category and $A$ is a Lagrangian algebra in $\CA$.
\item A morphism between two objects $(\CA,A)$ and $(\CB,B)$ is a pairs $(\CU,U)$, where $\CU$ is a closed multi-fusion $\CA$-$\CB$-bimodule and $U$ is a closed $A$-$B$-bimodule in the category of separable algebras in $\CU$.  
\item A 2-morphism between two morphisms $(\CU,U),(\CU',U') : (\CA,A)\to(\CB,B)$ is a quadruple $(\CP,p,w,f)$, where $(\CP,p):\CU\to\CU'$ is a 2-morphism in $\widehat\bfuscl$, 
$w$ is a left $U'\otimes_{A\boxtimes B}U$-module in $\CP$ (in particular, $w$ is a $U'$-$U$-bimodule), which is closed in the sense that the associated algebra homomorphism $U'\otimes_{A\boxtimes B}U \to [w,w]_{\CU'\boxtimes_{\overline{\CA}\boxtimes\CB}\CU^\rev}$ is an isomorphism,\footnote{Since $U'\otimes_{A\boxtimes B}U$ is a simple separable algebra in a matrix multi-fusion category, it has to be a matrix algebra \cite{kz4}. Therefore, $w$ exists and is unique up to isomorphism. Moreover, the algebra of left module maps of $w$ is isomorphic to $k$.}
and $f:p\to w$ is a morphism in $\CP$.
\item A 3-morphism between two 2-morphisms $(\CP,p,w,f),(\CP',p',w',f'):(\CU,U)\to(\CU',U')$ is an isomorphism class of triples $(H,\phi,t)$, where
$(H,\phi)$ represents a 3-morphism $(\CP,p)\to(\CP',p')$ in $\widehat\bfuscl$
and $t:H(w)\to w'$ is a left module isomorphism\footnote{By the previous footnote, $t$ is unique up to scalar.}
such that $t\circ H(f)=f'\circ \phi$.
Two triples $(H,\phi,t)$ and $(H',\phi',t')$ are isomorphic if there is a left module natural isomorphism $\eta:H\to H'$ such that $\phi=\phi'\circ\eta_p$ and $t=t'\circ\eta_w$.
\end{itemize}
\end{enumerate}

Note that for a 2-morphism $(\CP,p,w,f)$ in $\widehat\bfuscl_\bullet$ the items $\CP$ and $w$ are essentially redundant and for a 3-morphism $(H,\phi,t)$ the items $H$ and $t$ are essentially redundant. Moreover, two triples $(\id_\CP,\phi,\id_w)$ and $(\id_\CP,\phi',\id_w)$ represent the same 3-morphism $(\CP,p,w,f)\to(\CP,p',w,f')$ if and only if $\phi=\phi'$.
Therefore, the following assignment defines a 3-equivalence 
$$\widehat\FZ: \widehat\mfus_\bullet \to \widehat\bfuscl_\bullet$$
$$
\small 
\xymatrix{ (\CL,L) \ar@/^12pt/[rr]^{(\CX,x)} \ar@/_12pt/[rr]_{(\CX',x')} & \Downarrow (F,f) \overset{\phi}{\Rrightarrow} (F',f') \Downarrow 
& (\CM,M)}
\quad \longmapsto
$$
$$%\be \label{eq:3-functor*}
\small
\xymatrix{ 
  (\FZ(\CL),Z(L)) \ar@/^16pt/[rr]^{(\FZ^{(1)}(\CX),Z^{(1)}(x):=[x,x])} \ar@/_16pt/[rr]_{(\FZ^{(1)}(\CX'),Z^{(1)}(x'):=[x',x'])} & \Downarrow (\FZ^{(2)}(F),Z^{(2)}(f)) \overset{\FZ^{(3)}(\phi)}{\Rrightarrow} (\FZ^{(2)}(F'),Z^{(2)}(f')) \Downarrow 
& (\FZ(\CM)},Z(M))
$$%\ee
where $Z^{(2)}(f) = ([x,x'],\tilde f)$, $\FZ^{(3)}(\phi) = (\id,\phi,\id)$, 
and $[x,x']$ is defined by the adjunction
$$\Hom_{\Fun_{\CL|\CM}(\CX,\CX')}(G,[x,x']) \simeq \Hom_{{}_L\CX'_M}(G(x),x'),$$
and $\tilde f:F\to[x,x']$ is the mate of $f:F(x)\to x'$.
In particular, the 1-truncation of $\widehat\FZ$ recovers the pointed Drinfeld center 1-functor $\FZ: \mfus_\bullet \to \bfuscl$. We will discuss its physical meaning in the next two subsections.

\subsection{Boundary-bulk relation of 1+1D RCFT's}
Defects in quantum field theories (QFT) or condensed matter systems, such as boundaries and domain walls, have been becoming increasingly important in recent years. Among all these defects, 1-codimensional boundaries are especially important due to its defining roles in various holographic phenomena; 1-codimensional domain walls are also important because they encode some information of the intrinsic structures (such as dualities) of the physical system (see for example \cite{savit,ffrs,dkr1,kk}). Fusing two 1-codimensional domain walls along a non-trivial bulk QFT is an example of dimensional reduction processes in QFT's. It is a natural to ask how to compute such a fusion (see for example \cite{clswy} and references therein). Computing dimensional reduction processes amount to computing factorization homology \cite{lurie,af} in mathematics. Therefore, this question sits in the heart of both physics and mathematics.

Precise computation needs the precise mathematical descriptions of wall QFT's and bulk QFT's. They are not known for generic QFT's, but are known for some TQFT's and 1+1D RCFT's. Fusion of domain walls in TQFT's was studied in many works (see for example \cite{fsv,kwz}), and was explicitly computed for 2+1D anomaly-free topological orders \cite{kz1,bbj2,ai}. For 1+1D RCFT's, some partial results were known \cite{dkr}. 

\medskip
We briefly recall some basic results on RCFT's (see for example \cite{geometry} for a review). For a given modular-invariant bulk CFT $A_\bulk$, there is a family of boundary CFT's that are compatible with $A_\bulk$. All boundary CFT's are required to satisfy the so-called $V$-invariant boundary conditions (see for example \cite{opfa}), where $V$ is a rational vertex operator algebra (VOA), i.e. the category $\Mod_V$ of $V$-modules is a modular tensor category \cite{huang-mtc2}, and is called the chiral symmetry of the CFT. In this case, we have the following results. 
\bnu
\item The bulk CFT $A_\bulk$ is given by a Lagrangian algebra in the modular tensor category $\FZ(\Mod_V)=\Mod_V\boxtimes\overline{\Mod_V}$ \cite{kr2}. %, where $\FZ(\Mod_V)$ is the Drinfeld center of $\Mod_V$, $\boxtimes$ is Deligne's tensor product, and $\overline{\Mod_V}$ is the same monoidal category as $\Mod_V$ but equipped with braidings given by the anti-braidings of $\Mod_V$. 

\item A boundary CFT \cite{cardy} $A_\bdy$ is given by a simple special symmetric Frobenius algebra (SSSFA) in $\Mod_V$ \cite{fs,frs1,kr2}. It determines the bulk CFT $A_\bulk$ uniquely as its full center (see Definition \ref{def:left-right-full-centers}), i.e. $A_\bulk \simeq Z(A_\bdy)$ \cite{fjfrs,kr1,davydov}. Two boundary CFT's share the same bulk $A_\bulk$ if and only if they are Mortia equivalent as SSSFA's \cite{ffrs,kr1,davydov}. 

\item To each bulk CFT $A_\bulk$, there is a unique (up to equivalences) category $\CM_{A_\bulk}$ of boundary conditions. It is given by a unique (up to equivalences) indecomposable semisimple $\Mod_V$-module. For example, in the so-called Cardy case, $A_\bulk=Z(\one_{\Mod_V})$ (also called charge conjugate modular-invariant CFT), where $\one_{\Mod_V}=V$ is the tensor unit. In this case, $\CM_{A_\bulk}\simeq \Mod_V$ as left $\Mod_V$-modules. In general, if $A_\bulk=Z(A)$ for an SSSFA $A$ in $\Mod_V$, then $\CM_{A_\bulk} \simeq (\Mod_V)_A$ as left $\Mod_V$-modules. %, where $(\Mod_V)_A$ denotes the category of right $A$-modules in $\Mod_V$. 
An object in $\CM_{A_\bulk}$ is called a boundary condition of $A_\bulk$. 

\item Given a category of boundary conditions $\CM$, one can determine all the other ingredients of a RCFT via internal homs. More precisely, for $x\in\CM$, the boundary CFT associated to the boundary condition $x$ is given by the internal hom $[x,x]$ in $\Mod_V$. All $[x,x]$ for $x\in\CM$ share the same bulk CFT given by the full center $Z([x,x])=\int_{x\in\CM}[x,x]$ (recall Lemma\,\ref{thm:center-end} (4)). In other words, all these boundary CFT's are Morita equivalent. Moreover, for $x,y\in\CM$, the 0D domain wall between two boundary CFT's $[x,x]$ and $[y,y]$ is given by the internal hom $[x,y]$. 

\enu
%We will use $[x,x]$ to denote an internal hom or $[x,x]_{\Mod_V}$ to emphasize its ambient category. 

\medskip
For a fixed chiral symmetry $V$, it is still possible to have a few different bulk CFT's. They one-to-one correspond to the equivalence classes of Lagrangian algebras in $\FZ(\Mod_V)$. In this case, we consider 1D domain walls between different 2D bulk CFT's as depicted in Figure \ref{fig:cft-0}. Let $\CM_i$ be the category of boundary conditions of the bulk CFT's $A_\bulk^{(i)}$ for $i=1,2,3$. We have $A_\bulk^{(i)}:=[\id_{\CM_i},\id_{\CM_i}]_{\FZ(\Mod_V)}$, where $\id_{\CM_i}\in \Fun_{\Mod_V}(\CM_i,\CM_i)$ is the identity functor. It was shown in \cite{dkr} that 1D domain walls between two bulk CFT's $A_\bulk^{(1)}$ and $A_\bulk^{(2)}$ are given by the internal homs $[x,x]$ in $\FZ(\Mod_V)$ for $x\in \Fun_{\Mod_V}(\CM_1,\CM_2)$. All of these wall CFT's are Morita equivalent because, by the folding trick, they share the same bulk CFT given by $A_\bulk^{(1)} \boxtimes A_\bulk^{(2)}$, which should be viewed as a Lagrangian algebra in $Z(\Mod_V) \boxtimes \overline{Z(\Mod_V)}$. Moreover, for $x,y\in\Fun_{\Mod_V}(\CM_1,\CM_2)$, the 0D domain wall between two such wall CFT's $[x,x]$ and $[y,y]$ is precisely given by the internal hom $[x,y]$ in $\FZ(\Mod_V)$ \cite{dkr} as depicted in Figure \ref{fig:cft-0}. The horizontal fusion of two 0D walls $[x,y]$ and $[p,q]$ is given by the tensor product $[x,y]\otimes_{A_\bulk^{(2)}} [p,q]$ in $\FZ(\Mod_V)$ and can be computed by the following formula: 
\be \label{eq:dkr-formula}
[x,y]\otimes_{A_\bulk^{(2)}} [p,q] \simeq [p\circ x, q\circ y],
\ee
where $p\circ x, q\circ y\in\Fun_{\Mod_V}(\CM_1,\CM_3)$. 
Moreover, when $x=y$ and $p=q$, this isomorphism in (\ref{eq:dkr-formula}) is an algebra isomorphism. 

\begin{figure}%[H] 
\centering
    \includegraphics[scale=1]{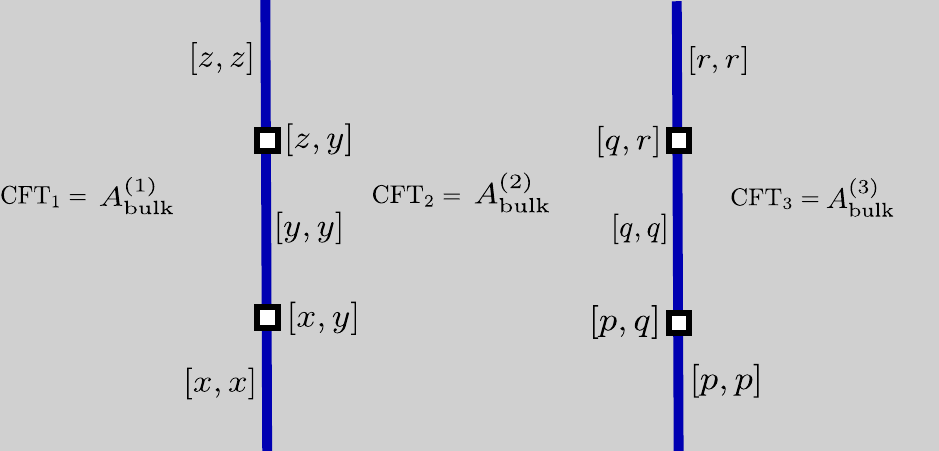}
    \caption{This figure depicts the 1+1D world sheet of three 1+1D bulk CFT's $A_\bulk^{(i)}$ separated by two 1D domain walls, each of which consists of three wall CFT's $[x,x],[y,y],[z,z]$ (resp. $[p,p],[q,q],[r,r]$) separated by two 0D domain walls $[x,y],[y,z]$ (resp. $[p,q],[q,r]$) for $x,y,z\in\Fun_{\Mod_V}(\CM_1,\CM_2), p,q,r\in \Fun_{\Mod_V}(\CM_2,\CM_3)$, where $\CM_i$ is the category of boundary conditions canonically associated to $A_\bulk^{(i)}$ for $i=1,2,3$. All internal homs live in $\FZ(\Mod_V)$. 
    } \label{fig:cft-0}
\end{figure}

\medskip
Note that the above results in \cite{dkr} are very limited because the chiral symmetry is fixed for all bulk CFT's. We would like to study the horizontal fusion of domain walls between different bulk CFT's equipped with different (but still rational) chiral symmetries. In order to see how to formulate the general situation mathematically, we reformulate the results of \cite{dkr} by Figure \ref{fig:cft-1} with new notations: 
$$
\CL:=\Fun_{\Mod_V}(\CM_1,\CM_1)^\rev, \quad \CM:=\Fun_{\Mod_V}(\CM_2,\CM_2)^\rev, \quad \CN:=\Fun_{\Mod_V}(\CM_3,\CM_3)^\rev
$$  
$$
\CX:=\Fun_{\Mod_V}(\CM_1,\CM_2), \,\,\, \CP:=\Fun_{\Mod_V}(\CM_2,\CM_3), 
\,\,\, \CU:=\Fun_{\CL|\CM}(\CX,\CX), \,\,\, \CV:=\Fun_{\CM|\CN}(\CP,\CP).
$$
By \cite{eo}, we have the following braided monoidal equivalences that form a commutative diagram: 
\be \label{eq:eo-isomorphisms}
\xymatrix@R=1.5em{ & & \FZ(\Mod_V) \ar[lld]_\simeq \ar[ld]^\simeq \ar[d]^\simeq \ar[rd]_\simeq \ar[rrd]^\simeq & & \\
\FZ(\CL) \ar[r]^\simeq & \CU & \FZ(\CM) \ar[l]_\simeq \ar[r]^\simeq & \CV & \FZ(\CN) \ar[l]_\simeq. 
}
\ee 
In particular, there are natural Morita equivalences among $\CL, \CM,\CN$ defined by the invertible $\CL$-$\CM$-bimodule $\CX$ and the $\CM$-$\CN$-bimodule $\CP$. The 1D domain walls $\CU, \CV$ are also invertible. Using the canonical equivalences in (\ref{eq:eo-isomorphisms}), 
\bnu
\item we can identify three bulk CFT's $A_\bulk^{(i)}$ with $Z(\one_\CL)=[\one_\CL,\one_\CL]_{\FZ(\CL)}$, $Z(\one_\CM)=[\one_\CM,\one_\CM]_{\FZ(\CM)}$, 
$Z(\one_\CN)=[\one_\CN,\one_\CN]_{\FZ(\CN)}$, respectively (see Corollary \ref{cor:center-A}); 
\item and identify the wall CFT's $[x,x],[p,p], \cdots$ and their 0D domain walls $[x,y], [p,q],\cdots$ in $\FZ(\Mod_V)$ with the same internal homs but living in $\CU$ or $\CV$ instead. \enu
As a consequence, we can relabeled Figure \ref{fig:cft-0} as Figure \ref{fig:cft-1}, which is ready to be generalized. %The data $(\CL,\one_\CL),(\CM,\one_\CM),(\CN,\one_\CN)$ and $x,y,z\in\CX,p,q,r\in\CP$ do not live on the 1+1D world sheet, but we add them to remind you the data in the domain category of $\FZ$. 

\begin{figure}%[H] 
\centering
    \includegraphics[scale=1]{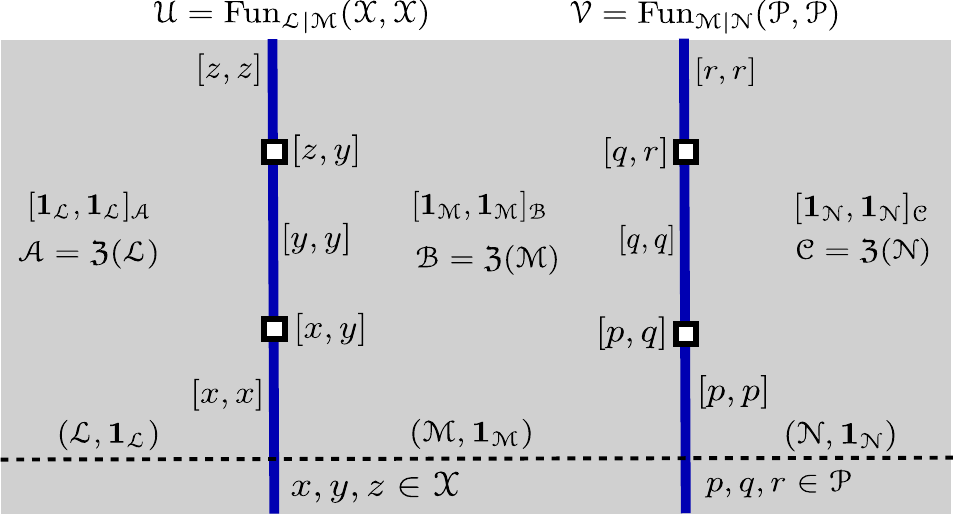}
    \caption{This figure depicts the 1+1D world sheet of three 1+1D bulk CFT's separated by two 1D domain walls (depicted as two vertical lines). %each of which consists of three wall CFT's separated by two 0D domain walls. 
    %$[x,y]_{\CU} \boxtimes_{[\one_\CM, \one_\CM]_\CB} [p,q]_{\CV} \simeq [x \boxtimes_{\CM} p, y \boxtimes_{\CM} q]_{\CU \boxtimes_{\CB}\CV}$ makes sense.
    } \label{fig:cft-1}
\end{figure}

\medskip
It is important to note that $\CL,\CM,\CN$ are only fusion categories. To generalize, we will treat $\CL,\CM,\CN$ as three generic (not necessarily Morita equivalent) fusion categories, regardless whether they are related to any VOA's. It is because all the internal hom constructions and the key formula (\ref{eq:dkr-formula}) are purely categorical results, which hold if we replace $\Mod_V$ by any fusion categories \cite{dkr}. Accordingly, we can generalize $\CX$ (resp. $\CP$) to be any semisimple (not necessarily invertible) $\CL$-$\CM$-bimodule (resp. $\CM$-$\CN$-bimodule). As a consequence, we obtain a generic situation depicted in Figure \ref{fig:cft-1}, in which the labels $(\CL,\one_\CL),(\CM,\one_\CM),(\CN,\one_\CN),x,y,z\in\CX,p,q,r\in\CP$ are not physical observables on the world sheet, but the data in the domain of $\FZ$. We add them along a dotted line to remind you the functoriality of $\FZ$. Also note that a generic object $(\CL,A)$ in the domain category of $\FZ$ is isomorphic to $({}_A\CL_A, \one_{{}_A\CL_A})$. Therefore, we obtain the physical meaning of the pointed Drinfeld center functor as depicted as follows 
\be \label{fig:bb-relation-CFT}
\raisebox{-4em}{\begin{picture}(220, 90)
   \put(0,10){\scalebox{1.8}{\includegraphics{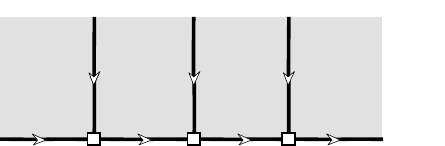}}}
   \put(-40,-55){
     \setlength{\unitlength}{.75pt}\put(-18,-19){
     \put(90, 98)       {\scriptsize $(\CL,L)$}
     \put(160, 98)       {\scriptsize $(\CM,M)$}
     \put(230, 98)       {\scriptsize $(\CN,N)$}
     \put(300, 98)       {\scriptsize $(\CO,O)$}
     \put(126,95)      {\scriptsize $(\CX,x)$}
     \put(195,95)      {\scriptsize $(\CY,y)$}
     \put(265,95)      {\scriptsize $(\CZ,z)$}
     \put(80, 160)    {\scriptsize $(\FZ(\CL),Z(L))$}
     %\put(85,140)  {\scriptsize $[\one_\CL,\one_\CL]_\CA$}
     \put(145, 160)    {\scriptsize $(\FZ(\CM),Z(M))$}
     %\put(160,140)  {\scriptsize $[\one_\CM,\one_\CM]_\CB$}
     \put(215, 160)    {\scriptsize $(\FZ(\CN),Z(N))$}
     %\put(240,140)  {\scriptsize $[\one_\CN,\one_\CN]_\CC$}
     \put(280, 160)    {\scriptsize $(\FZ(\CO),Z(O))$}
     %\put(315, 140) {\scriptsize $[\one_\CO,\one_\CO]_\CD$}
     \put(110,202)     {\scriptsize $(\FZ^{(1)}(\CX),[x,x])$}
     %\put(150,192) {\scriptsize $[x,x]_\CU$}
     \put(185,202)     {\scriptsize $(\FZ^{(1)}(\CY),[y,y])$}
     %\put(225,192) {\scriptsize $[y,y]_\CV$}
     \put(260,202)     {\scriptsize $(\FZ^{(1)}(\CZ),[z,z])$}
     %\put(300,192) {\scriptsize $[z,z]_\CW$}
     }\setlength{\unitlength}{1pt}}
  \end{picture}}
\ee

More precisely, let $U,V$ be rational VOA's and $\CL={_A}(\Mod_U)_A$, $\CM={_B}(\Mod_V)_B$ for some SSSFA's $A\in\Mod_U, B\in\Mod_V$. Then a SSSFA\footnote{A simple separable algebra in a fusion category can be endowed with a structure of SSSFA.} $L\in\CL$ (resp. $M\in\CM$) defines a boundary CFT of a modular-invariant bulk CFT $Z(L) \in \FZ(\CL) \simeq \FZ(\Mod_U)$ (resp. $Z(M)\in \FZ(\CM)\simeq \FZ(\Mod_V)$). Consider a 1D domain wall between two bulk CFT's $Z(L)$ and $Z(M)$. The 1+1D non-chiral symmetries $U\otimes_\Cb\overline{U}$ and $V\otimes_\Cb\overline{V}$ on the two sides of the wall are rational full field algebras \cite{hk-ffa}. By flipping the chirality of the anti-chiral parts and the orientations, as it was explained in \cite[Section\,5.4]{kz5}, this 0+1D domain wall can be viewed as a 0+1D gapless wall between two 1+1D chiral gapless boundaries (of the trivial 2+1D topological order) with the same 1+1D chiral symmetry $U\otimes_\Cb V$. In the neighborhood of the wall, the 1+1D chiral symmetry $U\otimes_\Cb V$ breaks down to a smaller 1+1D chiral symmetry $T^{(2)}$, which is still assumed to be rational. Moreover, there is a 0+1D chiral symmetry $T^{(1)}$, which is defined on the 0+1D wall and is an SSSFA in ${}_{U\otimes_\Cb V}(\Mod_{T^{(2)}})_{U\otimes_\Cb V}$ \cite{kz5}. The relation among $T^{(1)},U,V,T^{(2)}$ is illustrated by the following commutative diagram: 
$$
\xymatrix@R=1.5em{
& T^{(2)} \ar@{^{(}->}[ld] \ar@{^{(}->}[d] \ar@{^{(}->}[rd] & \\
U\otimes_\Cb V \ar[r] & T^{(1)} & U\otimes_\Cb V\, . \ar[l] 
}
$$
It was explained in \cite{kz5} that the 1D wall CFT's are objects in ${}_{T^{(1)}}({}_{U\otimes_\Cb V}(\Mod_{T^{(2)}})_{U\otimes_\Cb V})_{T^{(1)}}$, which is automatically a closed multi-fusion $\FZ(\CL)$-$\FZ(\CM)$-bimodule. Note that above discussion includes both the case $U=V$ (e.g. a 1D non-chiral trivial wall) and the case $U=\Cb$ (i.e. a 1D chiral wall) as special cases. There is no essential difference between chiral 0+1D walls and non-chiral 0+1D walls \cite{kz5}. 

For any semisimple $\CL$-$\CM$-bimodule $\CX$, there always exist (not necessarily unique) a rational VOA $T^{(2)}$ and an SSSFA $T^{(1)}$ in ${}_{U\otimes_\Cb V}(\Mod_{T^{(2)}})_{U\otimes_\Cb V}$ such that 
$$
\FZ^{(1)}(\CX)\simeq {_{T^{(1)}}}({}_{U\otimes_\Cb V}(\Mod_{T^{(2)}})_{U\otimes_\Cb V})_{T^{(1)}}
$$ 
as closed multi-fusion $\FZ(\CL)$-$\FZ(\CM)$-bimodules. Under this equivalence, $[x,x]$,$[y,y]$,$[z,z]$,$[x,y]$ and $[y,z]$ in $\FZ^{(1)}(\CX)$ for $x,y,z\in\CX$ can all be realized as 1D or 0D domain walls between two 1+1D bulk CFT's $Z(\one_{\CL})$ and $Z(\one_{\CM})$ \cite{kz5}.

In summary, the pointed Drinfeld center functor $\FZ$ gives the precise and rather complete boundary-bulk relation of 1+1D rational CFT's as illustrated in (\ref{fig:bb-relation-CFT}). In particular, the Formula (\ref{eqn:main}) tells us how to compute the fusion of two 1D (or 0D) wall CFT's along a non-trivial bulk 1+1D CFT.

\medskip
In general, there might be different choices of $(T^{(2)},T^{(1)})$, which define different 1D domain walls between $Z(\one_{\CL})$ and $Z(\one_{\CM})$ \cite{kz5}. It turns out that any two 1D walls 
$$
(\CU=\FZ^{(1)}(\CX),\, [x,x]_\CU),\quad\quad\quad (\CU' =\FZ^{(1)}(\CX'),[x',x']_\CV)
$$ 
between $Z(\one_{\CL})$ and $Z(\one_{\CM})$ as illustrated in Figure\,\ref{fig:2-category} can be obtained by taking different choices of the pairs $(T^{(2)},T^{(1)})$. Moreover, the 0D domain wall between $\CU$ and $\CU'$ in Figure\,\ref{fig:2-category} is precisely a 2-morphism in the codomain of $\widehat{\FZ}$:
$$
(\CW,F,[x,x']_\CW,\tilde{f}): (\FZ^{(1)}(\CX),Z^{(1)}(x)) \to (\FZ^{(1)}(\CX'),Z^{(1)}(x')),
$$ 
where $\CW:= \Fun_{\CL|\CM}(\CX,\CX')$, $F\in\CW$ and $\tilde f:F\to[x,x']_\CW$ is the mate of $f:F(x)\to x'$. In general, it is possible that the 1+1D chiral symmetries on $\CU$ and $\CU'$ are different. But, for the purpose of realizing $\CW$ physically, it is enough to consider the case in which their 1+1D chiral symmetries are the same $T^{(2)}$. In this case, we have 
$$
\FZ^{(1)}(\CX)\simeq {_{T^{(1)}_\CU}}({}_{U\otimes_\Cb V}(\Mod_{T^{(2)}})_{U\otimes_\Cb V})_{T^{(1)}_\CU}, \quad\quad\quad 
\FZ^{(1)}(\CX')\simeq {_{T^{(1)}_{\CU'}}}({}_{U\otimes_\Cb V}(\Mod_{T^{(2)}})_{U\otimes_\Cb V})_{T^{(1)}_{\CU'}}.
$$ 
Then $\CW$ can be realized by ${}_{T^{(1)}_{\CU'}}({}_{U\otimes_\Cb V}(\Mod_{T^{(2)}})_{U\otimes_\Cb V})_{T^{(1)}_\CU}$. In fact, we have a couple of morphisms 
\begin{align}
\tilde{f}_r: F \odot Z^{(1)}(x) &\xrightarrow{\tilde{f}\odot \id} [x,x']_\CW \odot [x,x]_\CU \to [x,x']_\CW, \label{eq:f-r} \\
\tilde{f}_l: Z^{(1)}(x')\odot F &\xrightarrow{\id\odot \tilde{f}} [x',x']_\CV\odot [x,x']_\CW \to [x,x']_\CW, \label{eq:f-l}
\end{align}
where the second unlabeled morphisms are naturally induced from the universal property of internal homs. They describe how the chiral fields in the wall CFT's $Z^{(1)}(x)$ and $Z^{(1)}(x')$ are mapped into the chiral fields in the 0D wall via the so-called operator product expansion (OPE). Therefore, this picture illustrates the physical meaning of the image of the 2-truncation of $\widehat{\FZ}$ in 1+1D rational CFT's. 

\begin{figure}[tb] 
\centering
    \includegraphics[scale=0.9]{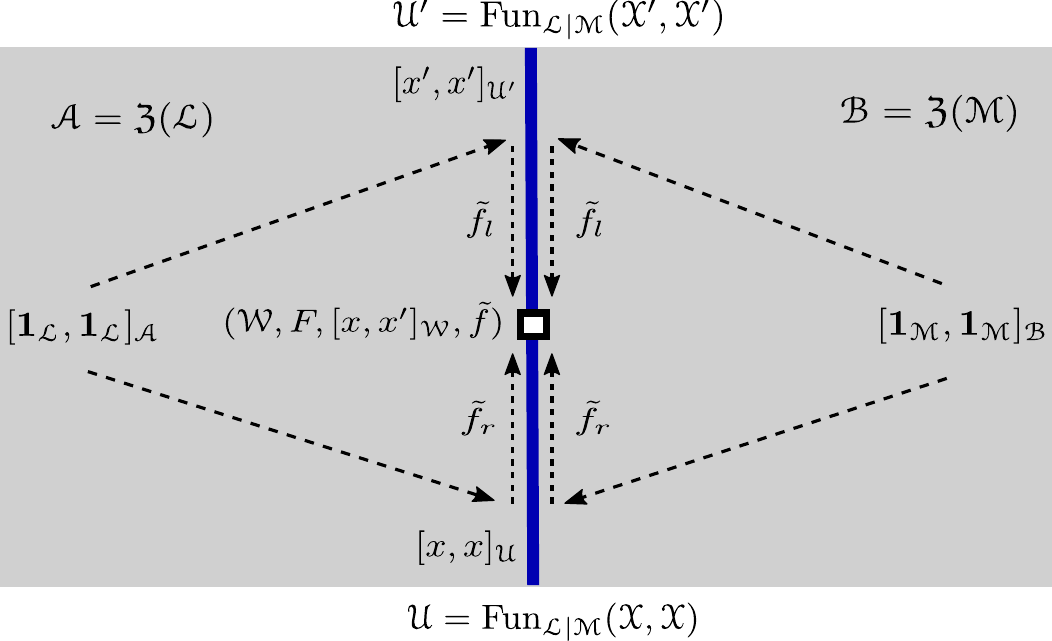}
    \caption{This figure illustrate the physical meaning of the image of 2-truncation of $\widehat\FZ$, where $\CW:= \Fun_{\CL|\CM}(\CX,\CX')$ and $F\in\CW$, and $f_r$ and $f_l$ are defined in (\ref{eq:f-r}) and (\ref{eq:f-l}), respectively.}
    \label{fig:2-category}
\end{figure}

\begin{rem}
Actually, by replacing $\CX$ in Figure\,\ref{fig:cft-1} by $\CX\oplus \CX'$, we can reduce the situation depicted in Figure\,\ref{fig:2-category} as a special case of the $\CU$-wall in Figure\,\ref{fig:cft-1}. 
\end{rem}

\begin{rem}
Using the physical meaning of $\FZ$, it is easy to see a physical proof of its faithfulness of $\FZ$ as illustrated by the following dimensional reduction process: 
$$
 \begin{picture}(100, 40)
   \put(-30,10){\scalebox{2.1}{\includegraphics{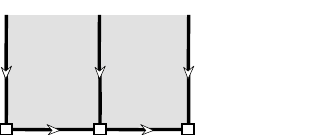}}}
   \put(-65,-55){
     \setlength{\unitlength}{.75pt}\put(-18,-19){
     \put(90, 98)       {\scriptsize $(\CL,L)$}
     \put(170, 98)       {\scriptsize $(\CM,M)$}
     \put(135,95)      {\scriptsize $(\CX,x)$}
     \put(213,95)      {\scriptsize $(\CM,M)$}
     \put(50,95)      {\scriptsize $(\CL,L)$}
     %\put(285,95)      {\scriptsize $(\CZ,z)$}
     \put(85, 160)    {\scriptsize $(\FZ(\CL),Z(L))$}
     %\put(85,140)  {\scriptsize $Z(L)$}
     \put(155, 160)    {\scriptsize $(\FZ(\CM),Z(M))$}
     %\put(160,140)  {\scriptsize $Z(M)$}
     %\put(240, 160)    {\scriptsize $\CC=\FZ(\CN)$}
     %\put(240,140)  {\scriptsize $[\one_\CN,\one_\CN]_\CC$}
     %\put(315, 160)    {\scriptsize $\CD=\FZ(\CO)$}
     %\put(315, 140) {\scriptsize $[\one_\CO,\one_\CO]_\CD$}
     \put(50,208) {\scriptsize $(\CL,L)$}
     \put(130,208)     {\scriptsize $(\CU,[x,x]_\CU)$}
     %\put(150,192) {\scriptsize $[x,x]_\CU$}
     \put(210,208)     {\scriptsize $(\CM,M)$}
     %\put(225,192) {\scriptsize $[y,y]_\CV$}
     %\put(285,212)     {\scriptsize $\CW=\FZ^{(1)}(\CZ)$}
     %\put(300,192) {\scriptsize $[z,z]_\CW$}
     }\setlength{\unitlength}{1pt}}
  \end{picture}
\quad\quad\quad\quad\quad
 \begin{picture}(10, 100)
   \put(40,10){\scalebox{2.1}{\includegraphics{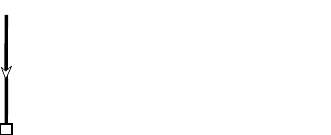}}}
   \put(5,-55){
     \setlength{\unitlength}{.75pt}\put(-18,-19){
     %\put(105, 98)       {\scriptsize $\CX$}
     \put(-50, 150) {\scriptsize $\xrightarrow{\mbox{dimensional reduction}}$}
     \put(63,95)      {\scriptsize $(\CX,x)$}
     %\put(285,95)      {\scriptsize $(\CZ,z)$}
     \put(53, 208)    {\scriptsize $\Fun_\bk(\CX,\CX)$}
     \put(75,180) {\scriptsize $[x,x]_{\Fun_\bk(\CX,\CX)}$}
     }\setlength{\unitlength}{1pt}}
  \end{picture}
$$
Note that $[x,x]_{\Fun_\bk(\CX,\CX)}$ determines $x\in\CX$ as the unique (up to equivalences) irreducible $[x,x]_{\Fun_\bk(\CX,\CX)}$-module in $\CX$, i.e. $\CX_{[x,x]}\simeq\bk$. This follows from the fact that $\CM_{[x,x]_{\Fun_\CC(\CM,\CM)}}\simeq\CC$ for a multi-tensor category $\CC$, a finite left $\CC$-module and $x\in\CM$. 
\end{rem}

\subsection{Topological Wick rotation and spatial fusion anomaly} \label{sec:fusion-anomaly}

Note that the physical meaning of Figure \ref{fig:cft-1} and Figure \ref{fig:bbr-gapped} are fundamentally different. On the one hand, Figure \ref{fig:cft-1} depicts the world sheet of 1+1D CFT's, where 1+1D is the spacetime dimension. On the other hand, Figure \ref{fig:bbr-gapped} depicts a physical configuration only in spatial dimension of 2d topological orders, where 2d is the spatial dimension. Naively, it seems that these two physical interpretations has no relation at all. Then it is natural to ask why the mathematical descriptions of Figure \ref{fig:cft-1} and Figure \ref{fig:bbr-gapped} are closely related? Is this just a meaningless and accidental coincidence in mathematics?

It turns out that there is a rather deep reason behind this coincidence. The mathematical theory of 1d gapless boundaries of 2d topological orders \cite{kz3,kz5} has revealed that one can naively "Wick rotating" a physical configuration of 2d topological orders in spatial dimension to obtain a 1+1D world sheet of a gapless boundary of a 2+1D topological order. This process is called a {\em topological Wick rotation} \cite{kz3}. In particular, if Figure \ref{fig:cft-1} is obtained from Figure \ref{fig:bbr-gapped} via a topological Wick rotation, then Figure \ref{fig:cft-1} should be viewed as the 1+1D world sheet of a 1d gapless boundary of the trivial 2d topological order, and the horizontal (resp. vertical) direction in Figure \ref{fig:cft-1} is the spatial (resp. temporal) direction.

\medskip
By the mathematical theory of 1d gapless boundaries of 2d topological orders in \cite{kz3,kz5}, the physical observables on 1+1D world sheet of a 1d boundary, viewed as a potentially anomalous 1d phase, form an enriched unitary fusion category. For example, when the left-most gray 2D region in Figure \ref{fig:cft-1} is viewed as the 1+1D world sheet of a 1d phase, which is anomaly-free in this case, all observables on it forms an $\CA$-enriched unitary fusion category ${^\CA}\CL$, whose underlying category is given by $\CL$ and the hom space $\hom_{{^\CA}\CL}(x,y)$ for $x,y\in \CL$ in the enriched category is given by the wall CFT $[x,y]_\CA$. Similarly, when the other two gray 2D regions in Figure \ref{fig:cft-1} are viewed as 1+1D world sheets of 1d phases, observables on them form enriched unitary fusion categories ${^\CB}\CM$ and ${^\CC}\CN$, respectively. The Drinfeld centers of ${^\CA}\CL$, ${^\CB}\CM$ and ${^\CC}\CN$ are all trivial (i.e. given by the category of finite dimension Hilbert spaces describing the trivial 2+1D bulk) \cite{kz2}. In the same spirit, when two blue lines in Figure \ref{fig:cft-1} are viewed as the 0+1D world line of two 0+1D phases, they are described by two enriched unitary categories ${^\CU}\CX$ and ${^\CV}\CP$, respectively. We have used the fact that $\CX$ (resp. $\CP$) is naturally a $\CU$-module (resp. $\CV$-module). Moreover, ${^\CU}\CX$ is naturally a ${^\CA}\CL$-${^\CB}\CM$-bimodule, and ${^\CV}\CP$ is naturally a ${^\CB}\CM$-${^\CC}\CN$-bimodule\footnote{Both bimodules are spatially invertible and define the spatial Morita equivalences among ${^\CA}\CL$, ${^\CB}\CM$, ${^\CC}\CN$ \cite{zheng,kz5}.}. 

It was shown in \cite{kz5} that the spatial (or horizontal) fusion of the two 0+1D phases ${^\CU}\CX$ and ${^\CV}\CP$ along the 1d phase ${^\CB}\CM$ is given by the relative tensor product over ${^\CB}\CM$, i.e.
\be \label{eq:fuse-enriched}
({^\CU}\CX) \boxtimes_{({^\CB}\CM)} ({^\CV}\CP) := {^{(\CU\boxtimes_\CB\CV)}}(\CX\boxtimes_\CM\CP). 
\ee
Physically, this fusion formula, together with the fact that ${^\CA}\CL, {^\CB}\CM, {^\CC}\CN$, regarded as 1d phases, and ${^\CU}\CX, {^\CV}\CP$, regarded as defects of codimension 1, are all anomaly-free, implies the vanishing of the spatial fusion anomaly (explained later). In particular, it means that there is a canonical isomorphism:
\be \label{eq:physics-formula}
[x,y]\otimes_{[\one_\CM,\one_\CM]_\CB} [p,q] \simeq [x\boxtimes_\CM p, y\boxtimes_\CM q],
\ee
where $[x,y]\otimes_{[\one_\CM,\one_\CM]_\CB} [p,q]$ denotes the horizontal fusion of 0D (or 1D if $x=y$) wall CFT's $[x,y]$ and $[p,q]$ in Figure \ref{fig:cft-1} and is defined by a coequalizer (recall (\ref{eq:coequalizer-B})). Mathematically, this formula (\ref{eq:physics-formula}) is proved rigorously in Theorem\,\ref{thm:main_formula} (1).

\begin{rem}
Note that the formula Eq.\,(\ref{eq:physics-formula}) can be applied to the general situations illustrated in Figure \ref{fig:2-category} because $[x,x]_\CU, [x,y]_\CW,[y,y]_\CV$ are just $[x,x]_\CO,[x,y]_\CO,[y,y]_\CO$, respectively, for $\CO:=\Fun_{\CL|\CM}(\CX \oplus \CY, \CX \oplus \CY)$. 
\end{rem}

We want to emphasize that the anomaly-free condition, i.e. Figure \ref{fig:cft-1} depicts the world sheet of a 1d boundary of the {\em trivial} 2d topological order, is crucial for the validity of the formula (\ref{eq:physics-formula}). In more general situations that appear in the mathematical theory of gapless boundaries, it is physically meaningful to replace the condition $\CA=\FZ(\CL),\CB=\FZ(\CM),\CC=\FZ(\CN)$ in Figure \ref{fig:cft-1} 
by the following weaker one: 
\begin{itemize}
\item[$(\star)$] $\CA,\CB,\CC$ are unitary modular tensor categories (also assume that $\CL,\CM,\CN$ are unitary fusion categories), there exist three unitary braided monoidal functors: 
\be \label{eq:anomalous-condition}
\CA \to \FZ(\CL), \quad\quad \CB \to \FZ(\CM), \quad\quad \CC \to \FZ(\CN),
\ee
which are potentially not equivalences. 
\end{itemize}
This replacement endows Figure \ref{fig:cft-1} with a similar physics meaning as the 1+1D world sheets of three different 1d boundaries (still described by ${^\CA}\CL, {^\CB}\CM, {^\CC}\CN$) of three potentially {\em non-trivial} 2+1D bulk topological orders. In this situation, we still have a natural morphism
$$
f: [x,y]\otimes_{[\one_\CM,\one_\CM]_\CB} [p,q] \to [x\boxtimes_\CM p, y\boxtimes_\CM q],
$$
which fails to be an isomorphism in general. Many examples of this failure are provided in \cite{kz3,kz5}. 

Fortunately and interestingly, the fusion formula (\ref{eq:fuse-enriched}) still holds in general situations. This is due to the so-called {\em Principle of Universality at a RG fixed point} proposed in \cite[Section\,6.3]{kz3}. It means that the direct spatial fusion given by $[x,y]\otimes_{[\one_\CM,\one_\CM]_\CB} [p,q]$ is not a RG stable and will flow to a fixed point given by $[x\boxtimes_\CM p, y\boxtimes_\CM q]$. In other words, after the spatial fusion, the system will flow to $[x\boxtimes_\CM p, y\boxtimes_\CM q]$ so that the formula (\ref{eq:fuse-enriched}) is preserved as the end of the RG flow. 

To some extent, the failure of $f$ being an isomorphism exactly catches the information of a non-trivial RG flow, and should be viewed as an indication of an anomaly, which was called {\em spatial fusion anomaly} and was introduced in \cite{kz3}. When ${^\CB}\CM$ is anomaly-free as a 1d phase (i.e. $\CB\simeq \FZ(\CM)$ \cite{kz2}), then $f$ is an isomorphism (proved in Theorem \ref{thm:main_formula}), i.e. the spatial fusion anomaly vanishes. Conversely, vanishing of the spatial fusion anomaly does not guarantee that ${^\CB}\CM$ is anomaly-free because there might be other anomalies. For example, the canonical chiral gapless boundary of a non-trivial chiral 2d topological order is anomalous as a 1d phase, but it has no spatial fusion anomaly as shown in \cite[Eq.\,(5.2)]{kz3}. For general chiral gapless boundaries, the spatial fusion anomaly does not vanish, but it vanishes for a subset of 0D and 1D wall CFT's on the 1+1D boundary. This subset has been identified in \cite[Remark\,6.3]{kz3}.

\appendix

\section{Davydov's definition of full center} \label{sec:davydov} 

Let $\CM$ be a monoidal category and $A$ an algebra in $\CM$. In \cite{davydov}, the full center $Z$ of $A$ is defined to be a pair $(Z,e)$, where $Z$ is an object in $\FZ(\CM)$ and $e: Z \to A$ is a morphism in $\CM$, such that it is terminal among all pairs $(X, f)$, where $X \in \FZ(\CM)$ and $f: X \to A$ is a morphism in $\CM$ such that the following diagram commutes:
\begin{align}\label{com_diag:X_A_comm}
    \xymatrix{
        X \otimes A \ar[r]^-{f\otimes \id_A} 
        \ar[d]_{\beta_{X,\, A}} & A \otimes A \ar[rd]^{m_A} & \\
        A\otimes X \ar[r]^-{\id_A\otimes f} & A\otimes A \ar[r]^{m_A} & A\, .
    }
\end{align}
It is known that $Z$ has a unique structure of an algebra in $\FZ(\CM)$ such that $e: Z \to A$ is an algebra homomorphism in $\CM$ (see \cite[Proposition 4.1]{davydov}).

In this appendix, we show that the full center defined by Davydov satisfies the universal property of the full center stated in Definition \ref{def:left-right-full-centers} (see also diagram \mref{diag:universal-property}).

\begin{lem}\label{lem:dav_lem_1}
Let $X$ be an algebra in $\FZ(\CM)$. If $f:X \to A$ is an algebra homomorphism making Diagram \eqref{com_diag:X_A_comm} commutative, then the composition $m: A \otimes X \xrightarrow{\id_A \otimes f} A \otimes A \xrightarrow{m_A} A$ is a unital $X$-action on $A$.       
\end{lem}

\begin{proof}
    Since $f$ is an algebra homomorphism, we have $u_A = f\circ u_X$. Thus the composition $A \simeq A\otimes\one_{\FZ(\CM)} \xrightarrow{\id_A \otimes u_X} A \otimes X \xrightarrow{m} A$ is $\id_A$. 
It is also easy to check that the commutativity of diagram \mref{com_diag:X_A_comm} implies that of the following one:
\begin{align*}
    \xymatrix{
        A \otimes X \otimes A \otimes X  \ar[r]^-{m\otimes m} 
        \ar[d]_{\id_A \otimes \beta_{X,A} \otimes \id_X} & A \otimes A \ar[rd]^{m_A} & \\
        A \otimes A \otimes X \otimes X  \ar[r]^-{m_A \otimes m_X} & A\otimes X \ar[r]^{m} & A\, .
    }
\end{align*}
Namely, $m$ is an algebra homomorphism hence is a unital $X$-action on $A$.
\end{proof}

\begin{lem}\label{lem:dav_lem_2}
    Let $X$ be an algebra in $\FZ(\CM)$. If $m: A \otimes X \to A$ is a unital $X$-action on $A$, then $f: X \simeq \one_\CM \otimes X \xrightarrow{u_A\otimes\id_X} A \otimes X \xrightarrow{m} A$ is an algebra homomorphism making Diagram \eqref{com_diag:X_A_comm} commutative.
\end{lem}

\begin{proof}
    Since $m$ is an algebra homomorphism, $f$ is also an algebra homomorphism. Note that
    \begin{align*}
        &m_A \circ (f \otimes \id_A) = m_A \circ(m \otimes m)\circ(u_A \otimes \id_X \otimes \id_A \otimes u_X) = m \circ \beta_{X, A},\\
    &m_A \circ (\id_A \otimes f)= m_A \circ (m \otimes m) \circ (\id_A \otimes u_X \otimes u_A \otimes \id_X) = m.
    \end{align*}
Therefore, Diagram \mref{com_diag:X_A_comm} is commutative.
\end{proof}

\begin{cor}
Let $X$ be an algebra in $\FZ(\CM)$. Giving a unital $X$-action on $A$ is equivalent to giving an algebra homomorphism $f:X \to A$ making Diagram \eqref{com_diag:X_A_comm} commutative
\end{cor}

\begin{prop}
Let $(Z,e)$ be the full center of $A$ in the sense of \cite{davydov}. Then $(Z,m)$ is the right center of $A$ in $\FZ(\CM)$ where $m$ is the composition $A \otimes Z \xrightarrow{\id_A \otimes e} A \otimes A \xrightarrow{m_A} A$ . 
\end{prop}

\begin{proof}
    By Lemma \ref{lem:dav_lem_1}, $m$ is a unital $Z$-action on $A$. Let $g: A \otimes X \to A$ be a unital $X$-action on $A$, where $X$ is an algebra in $\FZ(\CM)$. Then the composition $f: X \simeq \one_\CM \otimes X \xrightarrow{u_A} A \otimes X \xrightarrow{g} A$ is an algebra homomorphism making Diagram \eqref{com_diag:X_A_comm} commutative by Lemma \ref{lem:dav_lem_2}. By the universal property of $(Z,e)$, there exists a unique morphism $\underline{f}:X\to Z$ in $\FZ(\CM)$ such that $f = e \circ \underline{f}$. It remains to show that $\underline{f}$ is an algebra homomorphism. Since $e$ and $f$ are algebra homomorphism, we have
$$e\circ u_Z
= u_A
= f\circ u_X
= e\circ\underline{f}\circ u_X,$$
$$e \circ m_Z \circ (\underline{f} \otimes \underline{f}) 
= m_A \circ (e\otimes e) \circ (\underline{f} \otimes \underline{f})
= m_A \circ (f \otimes f)
= f \circ m_X 
= e \circ \underline{f} \circ m_X.$$
The universal property of $(Z,e)$ then implies that $u_Z = \underline{f}\circ u_X$ and $m_Z \circ (\underline{f} \otimes \underline{f}) = \underline{f} \circ m_X$. Namely, $\underline{f}$ is an algebra homomorphism.
\end{proof}

\end{document}